\documentclass[reqno,11pt]{amsart}
\usepackage{amsfonts, amsmath, amssymb, amsthm, color}
  \usepackage{amsfonts}
\usepackage{amssymb}
\usepackage{amsmath}
\usepackage{amsthm}
\usepackage{latexsym}
\usepackage{graphicx}
\usepackage{layout}
\usepackage{color}
\usepackage{geometry}
\usepackage{amsopn, enumitem}
\usepackage{palatino}

\usepackage{cite}

\allowdisplaybreaks[1]
\newtheorem{thm}{Theorem}[section]
\newtheorem{lemma}[thm]{Lemma}

\newtheorem{prop}[thm]{Proposition}
\newtheorem{rmk}[thm]{Remark}

\theoremstyle{definition}
 \numberwithin{equation}{section}

\newcommand{\rr}{\mathbb{R}}

\newcommand{\de}{\delta}
 \newcommand{\eps}{\varepsilon}
\newcommand{\la}{\lambda}
\newcommand{\D}{\Delta}
\newcommand{\R}{\mathbb{R}}

 \newcommand{\e }{\varepsilon }
   
   \newcommand{\into }{\int_{\Omega}}

\newcommand{\cal}{\mathcal }
   
   \newcommand{\N}{\mathbb{N}}
   \newcommand{\beq}{\begin{equation}}
   \newcommand{\eeq}{\end{equation}}
 
\def\bbm[#1]{\mbox{\boldmath $#1$}}
\newcommand{\bxi}{\boldsymbol\xi}

\DeclareMathOperator{\di}{d}

 \renewcommand{\(}{\left(}
\renewcommand{\)}{\right)}
\renewcommand{\[}{\left[}
\renewcommand{\]}{\right]}
\begin{document}

\title[Blowing-up solutions for the $SU(3)$ Toda System]{A continuum of solutions for the $SU(3)$ Toda System\\ exhibiting partial blow-up}
\author{Teresa D'Aprile}
\address[Teresa D'Aprile] {Dipartimento di Matematica, Universit\`a di Roma ``Tor
Vergata", via della Ricerca Scientifica 1, 00133 Roma, Italy.}
\email{daprile@mat.uniroma2.it}

\author{Angela Pistoia}
\address[Angela Pistoia] {Dipartimento SBAI, Universit\`{a} di Roma ``La Sapienza", via Antonio Scarpa 16, 00161 Roma, Italy}
\email{pistoia@dmmm.uniroma1.it}

\author{David Ruiz}
\address[David Ruiz]{Departamento de An\'{a}lisis Matem\'{a}tico, Granada, 18071 Spain.}
\email{daruiz@ugr.es}
\thanks{The authors have been supported by the Gruppo Nazionale per l'Analisi Matematica, la Probabilit\`a e le lore application (GNAMPA) of the Istituto Nazionale di Alta Matematica (IndAM).}
\thanks{The first and the second authors have been supported by the Italian PRIN Research Project 2012  \textit{Aspetti variazionali e perturbativi nei problemi differenziali nonlineari}.}
\thanks{The third author has been supported by the Spanish Ministry of
Science and Innovation under Grant MTM2011-26717 and by J.
Andalucia (FQM 116).}

\begin{abstract} In this paper we consider the so-called Toda System in planar domains under Dirichlet boundary condition. We show the existence of continua of solutions for which one component is blowing up at a certain number of points. The proofs use singular perturbation methods.

\bigskip



\end{abstract}
\maketitle

\section{Introduction}

In this paper we consider the following version of the $SU(3)$ Toda System on a smooth bounded domain $\Omega \subset \R^2$:

\begin{equation}\label{eq:e-1}
  \left\{
    \begin{array}{lr}
      - \D u_1 = 2 \rho_1 \frac{e^{u_1}}{\int_{\Omega}e^{u_1}} - \rho_2 \frac{e^{u_2}}{\int_{\Omega}e^{u_2}}, & x \in \Omega\\
     - \D u_2 = 2 \rho_2 \frac{e^{u_2}}{\int_{\Omega}e^{u_2}} - \rho_1 \frac{e^{u_1}}{\int_{\Omega}e^{u_1}},  & x \in \Omega \\
      \ u_1(x)=u_2(x)=0, & x \in \partial \Omega.
    \end{array}
    \right. \end{equation}

Here $\rho_1$, $\rho_2$ are positive constants. This problem, and its counterpart posed on compact surfaces or $\R^2$,  has been very much studied in the literature. The Toda system has a
close relationship with geometry, since it can be seen as the
Frenet frame of holomorphic curves in $\mathbb{CP}^N$ (see
\cite{guest}). Moreover, it arises in the study of the non-abelian
Chern-Simons theory in the self-dual case, when a scalar Higgs field
is coupled to a gauge potential, see \cite{dunne, tar, yys}.

It can also be seen as a natural generalization to systems of equations of the classical mean field equation. With respect to the scalar case, the Toda System presents some analogies but also some different aspects, which has attracted the attention of a lot of mathematical research in recent years. Existence for the Toda system has been studied from a variational point of view in \cite{bjmr, cheikh, jw, mruiz}, whereas blowing-up solutions have been considered in (\cite{ao, lyan, lwzao-GAFA, mpwei, osuzuki}), for instance.

The blow up analysis for the solutions to \eqref{eq:e-1} was performed in \cite{jlw}; let us explain it in some detail. Assume that $u_n=(u^1_n, u^2_n)$ is a blowing-up sequence of solutions of \eqref{eq:e-1} with $\rho_n=(\rho_n^1, \rho_n^2)$ bounded. Then, there exists a finite blow-up set $S=\{p_1, \dots p_k\} \subset \Omega$ such that the solutions are bounded away from $S$. Concerning the points $p_i$, let us define the local masses:

$$ \sigma^i = \lim_{r \to 0} \lim_{n \to +\infty} \rho_n^i \frac{\int_{B(p,r)} e^{u_n^i}}{\int_{\Omega} e^{u_n^i}}.$$

Then, the following scenarios are possible:

\begin{enumerate}
\item[a)] Partial blow-up: $(\sigma^1,\ \sigma^2)=(4\pi, 0)$. In such case, only one component is blowing up, and its profile is related to the entire solution of the Liouville problem.

\item[b)] Asymmetric blow-up: $(\sigma^1,\ \sigma^2)=(4\pi, 8\pi)$. In this case, both components blow up but the profile of each one is related to different scalar limit problems.

\item[c)] Full blow-up: $(\sigma^1, \ \sigma^2)=(8\pi, 8\pi)$. In this case, the profile of the sequence is described by an entire solution of the Toda System in $\R^2$, which have been classified in \cite{jw2}.
\end{enumerate}

As a consequence of their study, in \cite{jlw} it is proved that the set of solutions is compact for any ${\rho} \in  (\R^+)^2 \setminus \mathcal{C}$, where

$$ \mathcal{C} = \left (4 \pi \N \times \R^+ \right ) \cup \left (\R^+ \times 4 \pi \N \right ).$$

Here we are denoting $\R^+=[0,+\infty)$. In other words, if blow-up occurs, at least one component $u_n^i$ is quantized, and $\rho_n^i \to 4 k \pi$ for some $k \in \N$. This implies that the Leray-Schauder degree (in a sufficiently large ball) is well-defined and constant for any $\rho$ in a connected component of $\R^2 \setminus \mathcal{C}$. The computation of that degree is by now a widely open problem; as far as we know, the only result on the degree for the Toda System is \cite{mruiz2}. This is in contrast with the scalar mean field equation:

\begin{equation}\label{mf}
  \left\{
    \begin{array}{lr}
      - \D u = 2 \rho \frac{e^{u}}{\int_{\Omega}e^{u}}, &  x \in \Omega\\
           \ u(x)=0, & x \in \partial \Omega.
    \end{array}
    \right. \end{equation}

It is well-known that the set of solutions for \eqref{mf} is compact for $\rho \in \R \setminus 4 \pi \N$ (\cite{breme, lisha}), and an explicit formula for the degree in all compact settings was found in \cite{clin}.

Regarding the existence of blowing-up solutions for the Toda System, we are aware only of the results \cite{ao, mpwei, lyan}, which concern partial blow-up, asymmetric blow-up and full blow-up, respectively. In those papers, $\rho_n$ converges to a single point of $\mathcal{C}$.

The starting point of this paper is the following observation: \emph{in the Toda System one expects the existence of continua of families of blowing-up solutions}. Indeed, if the Leray-Schauder degree of two adjacent squares of $\R^2 \setminus \mathcal{C}$ is different, then there must be blowing-up solutions for all points ${\rho}$ in the common side. To see that, take ${\hat{\rho}}$, ${\tilde{\rho}}$ a point in each square, and $\gamma$ any curve joining them. If the set of solutions for \eqref{eq:e-1} were bounded for any ${\rho} \in \gamma$, then the homotopy invariance of the degree would imply that the degrees of the Toda System for ${\hat{\rho}}$, ${\tilde{\rho}}$ should coincide. Then, there must be blowing-up solutions for ${\rho_n}$ converging to a point of $\gamma \cap \mathcal{C}$ (see Figure 1). And this intersection could be any point of the common side.

\begin{figure}[h]
\centering
\includegraphics[width=0.6\linewidth]{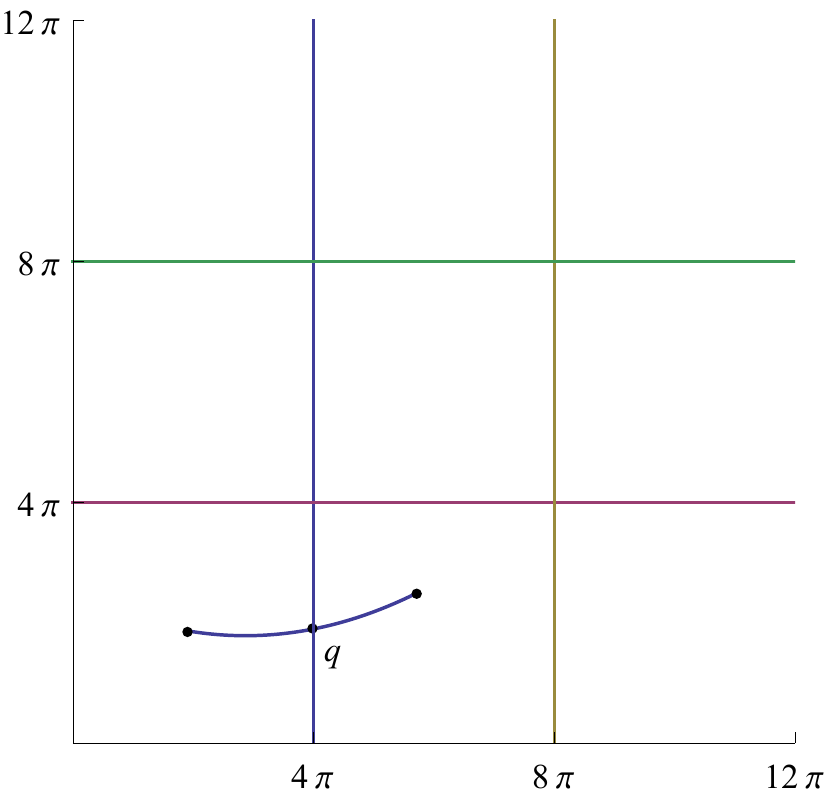}
\caption{There must be blowing-up solutions at the point $q$.}
\end{figure}

Let us show an easy example of such change of degree in adjacent squares. It is easy to show that the Leray-Schauder degree of \eqref{eq:e-1} is equal to $1$ if both $\rho_i \in [0, 4\pi)$ (see \cite{mruiz2}). Take now $\rho_2=0$; then, \eqref{eq:e-1} reduces to the classical mean field equation. If $\rho_1 \in (4 \pi, \ 8 \pi)$ and $\Omega$ is the disk, a classical Pohozaev-type argument shows that there is no solution. Therefore, the Leray-Schauder degree of \eqref{eq:e-1} is $0$ if $\Omega$ is the disk and ${\rho} \in (4\pi, 8 \pi) \times [0, 4\pi)$.

Since the degree does not depend on the metric, this change of degree remains for any simply connected domain $\Omega$. Then, for a simply connected domain \emph{there must be} blowing-up solutions $({u_n}, {\rho_n})$ with ${\rho_n}$ converging to any point of the segment $\{4 \pi \} \times (0,4\pi)$. One of the motivations of this paper is the study of the asymptotic behavior of those solutions.

\bigskip This paper is concerned with continua of solutions with a partial blow-up behavior. The case of asymmetric blow-up will be studied in a forthcoming paper. We do not expect the existence of continua of solutions exhibiting full-blow up.

More precisely, fixed $\rho_2 \in (0, 4\pi)$ and $k\in \N$, we look for blowing-up solutions $({u_n}, {\rho_n})$ with $\rho_n^1 \to 4 k \pi$, $\rho_n^2=\rho_2$.
Define:

$$ \mathcal F_k \Omega := \{ \bxi:=(\xi_1, \ \dots  \xi_k) \in \Omega^k:\ \xi_i \neq \xi_j \mbox{ if } i \neq j\}.$$

In order to describe the profile of our solutions, the following auxiliary problem appears quite naturally:

 \begin{equation}\label{ome}
\left \{ \begin{array}{lr}-\Delta z(x)  = 2 \rho_2 \frac{ h(x, \bxi) e^{z(x)}}{\int_{\Omega} h(y, \bxi) e^{z(y)}\, dy}   & x \in \ \Omega,\\ \ z(x)  =  0 & x \in \partial\Omega. \end{array} \right.
\end{equation}
with $\bxi \in \mathcal F_k \Omega$ and
 \begin{equation}\label{h}
h(x, \bxi)= \Pi_{i=1}^k |x-\xi_i|^2 e^{-4\pi H(x,\xi_i)}=e^{-4\pi \sum_{i=1}^k G(x,\xi_i)}.\end{equation}

Here $G$ is the Green function of the Laplace operator in $\Omega$ with Dirichlet boundary condition, and $H$ is its regular part. Equation \eqref{ome} is a singular mean field problem, where the function $h(\cdot, \bxi)$ vanishes exactly at the points $\xi_i$. This type of equation has been intensively studied in the literature; in this paper we will recall only the aspects which are needed in our analysis.

It is easy to show that equation \eqref{ome} is the Euler-Lagrange equation of the energy functional $I_{\bxi}: H_0^1(\Omega) \to \R$,

\begin{equation} \label{def I} I_{\bxi}(z)= \frac{1}{2} \int_{\Omega} |\nabla z|^2 \, dx - 2 \rho_2 \log \left ( \int_{\Omega} h(x,\bxi) e^{z(x)} \, dx \right).
\end{equation}
It is well-known that \eqref{ome} admits a solution for any $\rho_2 \in (0,4\pi)$. Moreover, if $\Omega$ is simply connected, the solution is unique and nondegenerate, see \cite{blin}. Then, the map $\bxi \mapsto z(\cdot,\bxi)$ is well-defined and smooth. Define:

\begin{equation} \label{Lambda} \Lambda(\bxi):= \frac{1}{2} I_{\bxi}(z(\cdot, \bxi)) -16 \pi^2 \left (\sum_{i=1}^k H(\xi_i,\xi_i) + \sum_{i \neq j} G(\xi_i, \xi_j) \right ). \end{equation}

Following \cite{li}, we shall say that a compact set $\cal{K} \in \mathcal F_k \Omega$ of critical points of $\Lambda$ is $C^1-$stable if, fixed a neighborhood $\cal{U}\supset \cal{K}$, any map $\Phi: \cal{U} \to \R$ sufficiently close to $\Lambda$ in $C^1$- sense has a critical point in $\cal{U}$.

Now we can state the main result of the paper:

\begin{thm} \label{teo} Let $k \in \N$, $\Omega$ a simply connected smooth domain, $\rho_2 \in (0,4\pi)$ and $\cal{K} \subset \mathcal F_k \Omega$ a $C^1$-stable set of critical points of $\Lambda$. Then there exists $\lambda_0>0$ such that, for any $\lambda \in (0,\lambda_0)$ there exist $(u_{\lambda}^1, \ u_{\lambda}^2)$ solutions of \eqref{eq:e-1} for $(\rho_{\lambda}^1, \ \rho_{\lambda}^2) $ satisfying that:

\begin{enumerate}

\item $\rho_{\lambda}^1 \to 4 \pi$ as $\lambda \to 0$, and $\rho_{\lambda}^2=\rho_2$.

\item There exists $\bxi=\bxi(\la) \in \cal{F}_k\Omega$, $\delta_i= \delta_i(\la) >0$ such that $d(\bxi, \cal{K})\to 0$, $\delta_i \to 0$ as $ \lambda \to 0$ and:

\begin{align*} u_{\lambda}^1(x) =- \frac 1 2 z(x,\bxi) + 2 \sum_{i=1}^k  \log \left ( \frac{1}{\delta_i^2 + |x-\xi_i|^2} \right ) + 8 \pi H(x,\xi) + o(1) \mbox{ in } H^1(\Omega) \mbox{ sense},
\\ u_{\lambda}^2(x)= z(x,\bxi) - \sum_{i=1}^k  \log \left ( \frac{1}{\delta_i^2 + |x-\xi_i|^2} \right )   -4 \pi H(x,\xi)  + o(1) \mbox{ in } H^1(\Omega) \mbox{ sense}. \end{align*}

\end{enumerate}

\end{thm}

Quite remarkably, we will also show that:

\begin{equation} \label{nabla} \partial_{\xi_j} \Lambda(\bxi) = 2 \pi \frac{ \partial z}{\partial x}(\xi_j, \bxi) - 32 \pi^2 \left( \frac{\partial H}{\partial x}(\xi_j, \xi_j) + \sum_{i\neq j} \frac{\partial G}{\partial x}(\xi_i, \xi_j)\right).  \end{equation}
Therefore, the point of concentration $\bxi$ makes \eqref{nabla} equal $0$ for any $\xi_j$. Observe that Theorem \ref{teo} is coherent with the blow-up analysis carried out in \cite[Theorem 2]{osuzuki}.

\medskip

In particular, if $k=1$, the map $\Lambda$ takes the form:

$$ \Lambda(\xi)= \frac{1}{2} I_{\xi}(z(\cdot, \xi)) -16 \pi^2  H(\xi,\xi).$$

We point out that $\Lambda$ diverges positively at $\partial \Omega$, so that the set:

$$K= \{ \xi \in  \Omega: \Lambda(\xi) = \inf \Lambda\}$$

is always a $C^1$-stable set of critical points of $\Omega$.

Furthermore, our result can be extended to more general frameworks. The assumptions $\rho_2 \in (0,4\pi)$ and $\Omega$ simply connected are imposed only to assure existence, uniqueness and non-degeneracy of the solutions of $\eqref{ome}$. For instance, if $\rho_2$ is sufficiently small, we can consider general domains. This is important because, for non simply connected domains, the function $\Lambda$ always have a $C^1$-stable set of critical points. All this will be commented in detail in Section 2.

The proof uses singular perturbation methods. The main difficulty here is that we are prescribing $\rho_2$, namely, the global mass of the second component. Therefore, our analysis is local (around the blow-up points for the first component) and also global (for the second component). Indeed, the location of the point depends on $\rho_2$ in a nontrivial manner, via the function $I_{\bxi}(z(\cdot, \bxi))$. When  $\rho_2=0$ one recovers well-known results for the scalar mean field equation (\cite{baraket, dkm, egpistoia}).

The rest of the paper in organized as follows. Section 2 is devoted to some preliminary results, like the study of the auxiliary problem \eqref{ome}. Moreover, a more general version of Theorem \ref{teo} is given. In Section 3 the problem is settled, and the approximate solutions are defined. In Section 4 we are concerned with the finite dimensional reduction; in particular, the error term and the linearized problem are studied. In Section 5 we obtain $C^1$ estimates of the reduced functional.

\medskip

\section{Preliminaries. Statement of the main Theorem}

In this section we first set the notation and basic well-known facts which will be of use in the rest of the paper. Afterwards, we shall study an auxiliary singular problem, whose solution will be used in further sections to build our approximating solution for the Toda system. Finally, we will state our main result, from which Theorem \ref{teo} follows.

In our estimates, we will frequently denote by $C>0$, $c>0$ fixed
constants, that may change from line to line, but are always
independent of the variable under consideration. We also use the
notations $O(1)$, $o(1)$, $O(\lambda)$, $o(\lambda)$ to describe the
asymptotic behaviors of quantities in a standard way.

We agree that  \begin{equation}\label{green}
G(x,y)={1\over 2\pi}\ln{1\over |x-y|}+H(x,y),\quad x,y\in\Omega
\end{equation}
is the Green function of the Laplacian operator in $\Omega$ under homogeneous Dirichlet boundary condition, and $H(x,y)$ is its regular part. 
\\

We shall write  $\|u\|:=\(\int\limits_{\Omega}|\nabla u|^2 dx\)^{1/2}$ to denote the norm in $H^1_0(\Omega)$ and
$\|u\|_p:=\(\int\limits_{\Omega}| u|^p dx\)^{1/p}$ for the usual norm in $L^p(\Omega)$, $1\leq p \leq +\infty$.  If $u=(u_1,u_2)$ is a vector function, we set
$\|u\|=\|u_1\|+\|u_2\| $ and $\|u\|_p=\|u_1\|_p+\|u_2\|_p $. We also denote by $\langle u,v\rangle= \int\limits_{\Omega} \nabla u\nabla v dx$ the usual inner product in $H^1_0(\Omega)$.
\\

Let us define the Hilbert spaces
\begin{equation}\label{ljs}
\mathrm{L} (\rr^2):= \mathrm{L}^2\left (\R^2, \left ({1\over 1+|y|^2} \right)^2 \, dy\right), \end{equation}
 and denote by $\| u\|_L$ its norm. Moreover, we define:
\begin{equation}\label{hjs}\mathrm{H} (\rr^2):=\left\{u\in {\rm W}^{1,2}_{loc}(\rr^2) \cap \mathrm{L}(\R^2) :\ \|\nabla u\|_{\mathrm{L}^2(\rr^2)}<+\infty\right\},\end{equation} with the norm
$$ \|u\|_{\mathrm{H}}:= \(\|\nabla u\|^2_{\mathrm{L}^2(\rr^2)}+ \|u\|^2 _{\mathrm{L}} \)^{1/2}.$$

Observe that those spaces are nothing but $L^2(\mathbb{S}^2)$ and $H^1(\mathbb{S}^2)$ when $u$ is interpreted as a function from the round sphere via the stereographic projection. In particular, there holds

\begin{prop}\label{compact}
The embedding $\mathrm{H} (\rr^2)\hookrightarrow\mathrm{L} (\rr^2)$
is compact.
\end{prop}

\medskip

We now state, for convenience of the reader, the well-known Moser-Trudinger inequality, see \cite{Moe,Tru}:
 \begin{lemma}\label{tmt} There exists $C>0$ such that for any bounded domain $\Omega$ in $\rr^2$
 $$\int\limits_\Omega e^{4\pi u^2/\|u\|^2}dx\le C |\Omega|,\ \hbox{for any}\ u\in{\rm H}^1_0(\Omega).$$
 In particular,  there exists $C>0$ such that for any $q\geq 1$,
 $$\| e^{u}\|_q\le  C |\Omega| e^{{q\over 16\pi}\|u\|^2},\
 \hbox{for any}\ u\in{\rm H}^1_0(\Omega).$$

The above inequality, for $q=1$, is usually written in the form:

\begin{equation} \label{onofri} \log \int_{\Omega} e^{u(x)}\, dx \leq \frac{1}{16 \pi} \|u\|^2 + \log |\Omega| + C. \end{equation}

 \end{lemma}

Let us now discuss some preliminary facts about the singular mean field problem \eqref{ome}. It is easy to show that equation \eqref{ome} is the Euler-Lagrange equation of the energy functional $I_{\bxi}$ defined in \eqref{def I}.
By \eqref{onofri}, $I_{\bxi}$ is coercive for all $\rho_2 \in (0,4\pi)$, and hence a solution to \eqref{ome} is found as a global minimizer.

In what follows we will assume the following hypothesis on problem \eqref{ome}:

\begin{enumerate}[label=(H), ref=(H)]

\item \label{H} There exists an open set $\cal{D}  \subset \mathcal F_k \Omega $ such that for any $\bxi \in \cal{D}$, problem \eqref{ome} is uniquely solvable and the solution is nondegenerate. In other words, if we denote by $z(x,\bxi)$ the unique solution of \eqref{ome}, then the linear problem
 $$\left \{ \begin{array}{l}-\Delta \psi(x)  = 2\rho_2 \left ( \displaystyle \frac{ h(x, \bxi) e^{z(x,\bxi)}\psi(x)}{\int_{\Omega} h(y, \bxi) e^{z(y,\bxi)}\, dy} - \frac{h(x, \bxi) e^{z(x,\bxi)}\int_{\Omega} h(y, \bxi) e^{z(y,\bxi)}\psi(y)\, dy}{\(\int_{\Omega} h(y, \bxi) e^{z(y,\bxi)}\, dy\)^2} \right )\ \ \hbox{ in}\ \Omega, \\ \ \psi(x)  =0 \ \ \ \hbox{ on}\ \partial\Omega \end{array} \right.$$
 admits only the trivial solution.
\end{enumerate}

Assumption \ref{H} seems very restrictive, but actually it holds for any simply connected domain $\Omega$, $\rho_2 \in (0,4\pi)$ and $\cal{D} = \mathcal F_k \Omega$, see \cite{blin}. Moreover, it holds also for any arbitrary domain $\Omega$ and $\rho_2$ sufficiently small, as shown in the following proposition.

\begin{prop} \label{piccolo}  Given $\cal{D} \subset \overline{D} \subset \mathcal F_k \Omega$ there exists $\delta>0$ such that if $\rho_2 \in (0, \delta)$, condition \ref{H} is satisfied.

\end{prop}

\begin{proof} We first prove uniqueness. Reasoning by contradiction, assume that there exists $\rho_n \to 0$, $\bxi_n \in D$, and $z^0_n$, $z^1_n$ two solutions for the problem:

 \begin{equation} \label{ome3}
 \begin{array}{l}-\Delta z(x)  = 2 \rho_n \frac{ h(x, \bxi_n) e^z(x)}{\int_{\Omega} h(y, \bxi_n) e^{z(y)}\, dy}   \ \mbox{in}\ \Omega,\\ \ z(x)  =  0\ \mbox{on}\ \partial\Omega, \end{array}
\end{equation}
with $h$ defined as in \eqref{h}.

By compactness of solutions (see \cite{btar}), we can pass to a subsequence so that $\xi_n \to \xi_0 \in \overline{\cal{D}}$, $z_n^i \to z^i$ in $C^{2,\alpha}$ sense. Passing to the limit in \eqref{ome3}, we conclude that $z^i =0$ for $i=0$, $1$.

Let us define:

$$ \Psi: C_0^{2,\alpha}({\Omega}) \times \mathcal F_k \Omega  \times \R \to C^{0,\alpha}({\Omega}),$$
\begin{equation}
 \Psi(z,\bxi,\rho)=-\Delta z(x) - 2 \rho \frac{ h(x, \bxi) e^z}{\int_{\Omega} h(x, \bxi) e^z\, dx},
\end{equation}
where $ C_0^{2,\alpha}({\Omega})$ is the space of H\"older $C^2$ functions that vanish at $\partial \Omega$.

Clearly $z$ is a solution of \eqref{ome} if and only if $\Psi(z, \bxi, \rho)=0$. Clearly, $\Psi(0,\bxi_0, 0)=0$, and moreover

$$D \Psi_{(0,\bxi_0,0)}(\phi, 0, 0)=- \Delta \phi,$$ which is an invertible operator from $ C_0^{2,\alpha}({\Omega})$  to $ C^{0,\alpha}({\Omega})$ . Therefore, we are under the conditions of the Implicit Function Theorem. There exist $r>0$, $\delta>0$, $U$ a neighborhood of $0$ in $C_0^{2,\alpha}({\Omega})$ and a $C^1$ map:
$$z: B(\bxi_0, r) \times (-\delta, \delta) \to C_0^{2,\alpha}({\Omega}),$$
$$ (\boldsymbol{\tilde{\xi}}, \tilde{\rho}) \mapsto z_{\boldsymbol{\tilde{\xi}}, \tilde{\rho}},$$
such that $\Psi( \boldsymbol{\tilde{\xi}}, \tilde{\rho}, z_{\boldsymbol{\tilde{\xi}}, \tilde{\rho}})=0$, and $z_{\boldsymbol{\tilde{\xi}}, \tilde{\rho}}$ is the unique function in $U_\xi$ satisfying that.

Now, if $n$ is large enough, both $z^i_n$ belong to $U$ and $\bxi_n \in B(\xi_0,r)$. Therefore, $z_n^0=z_n^1$.

\medskip

Let us now prove non-degeneracy. Again by contradiction, assume $\rho_n \to 0$, $\bxi_n \in D$, $z_n$ a solution of \eqref{ome3} and $\psi_n$ a nontrivial solution of the problem:

\begin{equation} \label{nd2} \begin{array}{l}-\Delta \psi_n(x)  = f_n(x)\ \ \hbox{ in}\ \Omega, \\ \ \psi(x)  =0 \ \ \ \hbox{ on}\ \partial\Omega \end{array} \end{equation}

with

$$f_n(x)= 2\rho_n \left ( \displaystyle \frac{ h(x, \bxi_n) e^{z_n(x)}\psi_n(x)}{\int_{\Omega} h(y, \bxi_n) e^{z_n(y)}\, dy} - \frac{h(x, \bxi_n) e^{z_n(x)}\int_{\Omega} h(y, \bxi_n) e^{z_n(y)}\psi_n(y)\, dy}{\(\int_{\Omega} h(y, \bxi_n) e^{z_n(y)}\, dy\)^2} \right ).$$

Since the above is a linear problem, we can normalize $\psi_n$ so that $\|\psi_n\|_{\infty}=1$. As above, we have that $z_n \to 0$ in $C^{2,\alpha}$ sense. Then, $\|f_n\|_{\infty} \to 0$, which is a contradiction with \eqref{nd2}.

\end{proof}

The following lemma will be useful in what follows:

\begin{lemma} \label{zeta} Under assumption \ref{H}, the map $\mathcal F_k \Omega \ni \bxi \mapsto z(\cdot,\bxi)$ is a $C^1$ map from $\cal{D}$ to $C^{2,\alpha}({\Omega})$. \end{lemma}

\begin{proof} The argument follows the same ideas of the proof of Proposition \ref{piccolo}. Define again the map $\Psi$, but now without dependence on $\rho$, namely:

$$ \Psi: C_0^{2,\alpha}({\Omega}) \times \mathcal F_k \Omega  \to C^{0,\alpha}({\Omega}),$$
\begin{equation} \label{ome2}
 \Psi(z,\xi)=-\Delta z(x) - 2 \rho \frac{ h(x, \bxi) e^z}{\int_{\Omega} h(x, \bxi) e^z\, dx}.
\end{equation}

\medskip Clearly $z$ is a solution of \eqref{ome} if and only if $\Psi(z, \bxi, \rho)=0$.

Assume condition \ref{H}, fix some $\bxi_0 \in \cal{D}$ and take $z$ the unique solution to \eqref{ome}. Clearly, $\Psi(z,\bxi_0, \rho)=0$, and moreover:

$$D \Psi_{(z,\bxi_0)}(\psi, 0)=  -\Delta (\psi) - K(\psi),$$
where the operator $K:C_0^{2,\alpha} \to C^{0,\alpha}$ is defined as:

$$ K(\psi)=2\rho_2 \left (\displaystyle \frac{ h(x, \bxi) e^{z(x,\bxi)}\psi(x)}{\int_{\Omega} h(y, \bxi) e^{z(y,\bxi)}\, dy} - \frac{h(x, \bxi) e^{z(x,\bxi)}\int_{\Omega} h(y, \bxi) e^{z(y,\bxi)}\psi(y)\, dy}{\(\int_{\Omega} h(y, \bxi) e^{z(y,\bxi)}\, dy\)^2} \right ).$$

Observe now that $\Delta:C_0^{2,\alpha} \to C^{0,\alpha} $ is an isomorphism and $K$ is a compact operator. Therefore $- \Delta -K$ is a Fredholm operator with $0$ index. By assumption \ref{H} $D \Psi_{(z,\xi)}(\psi, 0)$ has zero kernel, and hence it is a bijection.

Therefore we are under the conditions of the Implicit Function Theorem, and
there exists $r>0$ and a $C^1$ map:
$$z: B(\bxi_0, r)  \to C_0^{2,\alpha}({\Omega}),$$
$$ \boldsymbol{\tilde{\xi}}\mapsto z_{\boldsymbol{\tilde{\xi}}},$$
such that $z(\bxi_0)=z$ and $\Psi( \boldsymbol{\tilde{\xi}}, z_{\boldsymbol{\tilde{\xi}}})=0$. This concludes the proof, since the solutions of \eqref{ome} are unique by assumption \ref{H}.

\end{proof}

\medskip

By the previous lemma, the map $\Lambda : \cal{D} \to \R$ defined as in \eqref{Lambda} is $C^1$. The main result of this paper is the following:

\begin{thm} \label{main} Assume that $\Omega$, $\rho_2>0$ satisfy assumption \ref{H}. Let $\cal{K} \in \cal{D}$ be a $C^1$-stable set of critical points of $\Lambda$. Then there exists $\lambda_0>0$ such that, for any $\lambda \in (0,\lambda_0)$ there exist $(u_{\lambda}^1, \ u_{\lambda}^2)$ solutions of \eqref{eq:e-1} for $(\rho_{\lambda}^1, \ \rho_{\lambda}^2) $ satisfying that:
\begin{enumerate}

\item $\rho_{\lambda}^1 \to 4 \pi$ as $\lambda \to 0$, and $\rho_{\lambda}^2=\rho_2$.

\item There exists $\bxi=\bxi(\la) \in \cal{F}_k\Omega$, $\delta_i= \delta_i(\la) >0$ such that $d(\bxi, \cal{K}) \to 0$, $\delta_i \to 0$ as $ \lambda \to 0$ and:

\begin{align*} u_{\lambda}^1(x) =- \frac 1 2 z(x,\bxi) + 2 \sum_{i=1}^k  \log \left ( \frac{1}{\delta_i^2 + |x-\xi_i|^2} \right ) + 8 \pi H(x,\xi) + o(1) \mbox{ in } H^1(\Omega) \mbox{ sense},
\\ u_{\lambda}^2(x)= z(x,\bxi) - \sum_{i=1}^k  \log \left ( \frac{1}{\delta_i^2 + |x-\xi_i|^2} \right )   -4 \pi H(x,\xi)  + o(1) \mbox{ in } H^1(\Omega) \mbox{ sense}. \end{align*}

\end{enumerate}

\end{thm}

Therefore, the point of concentration converges to a critical point of $\Lambda$. In next lemma we show that the partial derivatives of $\Lambda$ admit the expression \eqref{nabla}.

\begin{lemma} \label{grad} For every $\bxi\in{\cal F}_k(\Omega)$, equality \eqref{nabla} holds.
\end{lemma}
\begin{proof}

The derivatives of the terms $H(\xi_i, \xi_i)$ and $G(\xi_i,\xi_j)$ follow from the symmetry of those functions, so we need to show that:

\begin{equation} \label{hi} \partial_{\xi_j} \tilde{I}(\bxi) = 4 \pi \frac{ \partial z}{\partial x}(\xi_j, \bxi),\end{equation}
where
\begin{equation} \label{itilde} \tilde{I}: \mathcal{D} \to \R,\ \tilde{I}( \bxi)= I_{\bxi}(z(\cdot, \bxi)). \end{equation}

From the representation formula we have $$z(x,\bxi)=\frac{2 \rho_2}{\into e^{-4\pi\sum_{h=1}^k G(y,\xi_h)+z(y,\bxi)}dy}\into G(x,y)e^{-4\pi\sum_{h=1}^k G(y,\xi_h)+z(y,\bxi)}dy.$$ Denoting by  $\partial_1 G$ and $\partial_2 G$  the derivative of $(x,y)\mapsto G(x,y)$ with respect to the first variable and the second variable, respectively, we compute
$$ \nabla z(\cdot ,\bxi)\big|_{x=\xi_j}=\frac{2 \rho_2}{\into e^{-4\pi\sum_{h=1}^k G(y,\xi_h)+z(y,\bxi)}dy}\into \partial_1 G(\xi_j, y)e^{-4\pi \sum_{h=1}^kG(y,\xi_h)+z(y,\bxi)}dy
$$

Since $z(\cdot, \bxi)$ is a solution of \eqref{ome}, then $I'_{\bxi}(z)(\partial_{\xi_j} z)=0$, so we have that:

$$\partial_{\xi_j} \tilde{I}(\bxi) = \frac{(-2 \rho_2)(-4\pi)}{\into e^{-4\pi  \sum_{h=1}^k G(y,\xi_h)+z(y,\bxi)}} \into  e^{-4\pi  \sum_{h=1}^kG(y,\xi_h)+z(y,\bxi)}\partial_2 G(y,\xi_j)dy.$$

The conclusion follows from the symmetry of $G$.
\end{proof}

Finally, next lemma is devoted to give conditions that guarantee the existence of $C^1$-stable set of critical points of $\Lambda$.

\begin{lemma} The map $\Lambda$ admits a $C^1$-stable set of critical points provided that either

\begin{enumerate}

\item[1)] $\Omega$ is simply connected, $k=1$ and $\rho_2 \in (0,4\pi)$.

\end{enumerate}

or

\begin{enumerate}

\item[2)] $\Omega$ is not simply connected, $k$ is arbitrary and $\rho_2 \in (0,\eps)$ for some $\eps \in (0, 4\pi)$.

\end{enumerate}

\end{lemma}

\begin{proof} As commented previously, condition \ref{H} is satisfied in case 1) by \cite{blin} (actually this is true for any $k$, and $\cal{D}= \cal{F}_k(\Omega)$). Moreover, \ref{H} holds in case 2) by Proposition \ref{piccolo}, for any $\cal{D} \subset \overline{\cal{D}} \subset \cal{F}_k\Omega$. In particular, the function $\Lambda$ is well defined.

The proof is based on the following claim: there exists $C>0$ independent of the choice of $\mathcal{D}$ such that
$$ \| \tilde{I}\|_{C^1(\cal{D})} \leq C,$$
where $\tilde{I}$ is defined in \eqref{itilde}. Indeed, for any $z \in H_0^1(\Omega)$,

$$ I_{\bxi}(z) \geq \frac{1}{2} \int_{\Omega} |\nabla z|^2 \, dx - 2 \rho_2 \log \left ( \int_{\Omega} e^{z(x)} \, dx \right) \geq C,$$
by the Moser-Trudinger inequality (Lemma \ref{tmt}). Moreover, since $z(\cdot, \bxi)$ is a minimizer for $I_{\bxi}$,
\begin{align*} I_{\bxi}(z(\cdot, \bxi)) \leq I_{\bxi}(0) & = -\rho_2 \log \int_{\Omega}e^{-4 \pi \sum_{i=1}^k G(x, \xi_i)}\,dx  \\ & \leq \rho_2 \left( \frac{1}{|\Omega|} \int_{\Omega} 4 \pi \sum_{i=1}^k G(x, \xi_i)\, dx -\log |\Omega| \right) \leq C,\end{align*}
where we have used Jensen inequality. This concludes the $C_0$ estimate. Moreover, by Lemma \ref{tmt} the functional $I_{\bxi}$ is coercive. The $C^0$ estimate of $\tilde{I}$ implies in particular that $\|z(\cdot, \bxi) \| \leq C$. Standard regularity arguments imply that $\|z(\cdot, \bxi) \|_{C^1} \leq C$. The $C^1$ estimate of $\tilde{I}$ follows then by \eqref{hi}.

\medskip

In case 1), $\Lambda$ takes the form:

$$ \Lambda(\xi)= \frac{1}{2} I_{\xi}(z(\cdot, \xi)) -16 \pi^2  H(\xi,\xi).$$

Since $H(\xi, \xi) \to +\infty$ as $\xi \to \partial \Omega$, we conclude that the set

$$K= \{ \xi \in  \Omega: \Lambda(\xi) = \inf \Lambda\}$$

is always a $C^1$-stable set of critical points of $\Omega$.

In case 2), a min-max scheme was developed in \cite{dkm} for the function:

$$ \bxi \mapsto -16 \pi^2 \left (\sum_{i=1}^k H(\xi_i,\xi_i) + \sum_{i \neq j} G(\xi_i, \xi_j) \right ).$$

Observe that such scheme depends on the behavior of $H$, $G$ when the points $\xi_i$ coincide or tend to the boundary of $\Omega$. Therefore, the argument is not affected by a bounded $C^1$ perturbation of the function, and is valid also for $\Lambda$.

\end{proof}

\begin{rmk}
Let us point out that the singular mean field problem \eqref{ome} admits also a solution for $\rho_2 \in (4\pi, 8 \pi)$ and any $\xi \in \Omega$, as proved in \cite{mr}. However, it is not clear at all wether condition \ref{H} is satisfied or not in such case. If it is satisfied, then Theorem \ref{main} is also applicable in this situation.

If $\rho_2 > 8 \pi$ there may be no solutions for \eqref{ome}. This is the case of the disk and $\xi=0$, as shown in \cite{bmalchiodi} (Section 5.3). But, again, if \ref{H} is satisfied, the assertion of Theorem \ref{main} holds.

\end{rmk}

\section{Setting of the problem}

In order to prove Theorem \ref{main}, we introduce the system:
\begin{equation}\label{s}
\left\{\begin{array}{l}
\Delta u_1+ 2\lambda e^{u_1}-\rho_2 \frac{e^{u_2}}{\int_{\Omega}e^{u_2}}=0 \ \mbox{in} \ \Omega, \\
\Delta u_2+ 2 \rho \frac{e^{u_2}}{\int_{\Omega}e^{u_2}}- \lambda e^{u_1}=0 \ \mbox{in} \ \Omega, \\
u_1=u_2 =0 \ \mbox{on} \ \partial \Omega.
\end{array}
\right.
\end{equation} where $\lambda>0$ will be chosen sufficiently small.\\

First let us rewrite problem \eqref{s} in a more convenient way. For any $ p>1,$ let $i^*_p:L^{p}(\Omega)\to H^1_0(\Omega)$ be the adjoint operator of the embedding $i_p:H^1_0(\Omega)\hookrightarrow L^{p\over p-1 }(\Omega),$ i.e. $u=i^*_p(v)$ if and only if $-\Delta u=v$ in $\Omega,$ $u=0$ on $\partial\Omega.$
We point out that $i^*_p$ is a continuous mapping, namely
\begin{equation}
\label{isp}
\|i^*_p(v)\|_{ H^1_0(\Omega)}\le c_p \|v\|_{p}, \ \hbox{for any} \ v\in L^{p}(\Omega),
\end{equation}\\
for some constant $c_p$ which depends on $\Omega$ and $p.$

We introduce the following notation.  We denote by ${     u}:=(u_1,u_2)$ and set $   {i^*_p}({     v}):=\(i^*_p(v_1),i^*_p(v_2)\).$
Then problem \eqref{s} is equivalent to problem
\begin{equation}\label{ps}
\left\{\begin{aligned}&{     u}=   {i^*_p}\({     F} ({     u})\)  ,\\ & {     u}\in H^1_0(\Omega)\times H^1_0(\Omega).\\\end{aligned}\right.
\end{equation}
where
$$
{     F} ({     u}):=\(2\lambda f(u_1)-\rho g(u_2) ,   2\rho_2 g(u_2)- \lambda f(u_1)\)$$
and
$$
  f (u_1):=e^{u_1}\  \hbox{and}\  g(u_2):=  \frac{e^{u_2}}{\int\limits_{\Omega}e^{u_2(x)}dx}.
$$
Let us introduce the bubbles
$$
w _{\de,\xi}(x):=\log  {8\de^2\over\(\de^2+|x-\xi|^2\)^2}=-2\log\de+w\({x-\xi\over\de}\)\, x,\xi\in\rr^2,\ \de>0$$
where
 \begin{equation}\label{w}w(y):=\log  {8\over\(1+|y|^2\)^2},
\end{equation}
which solve the
Liouville problem
\begin{equation}\label{plim}
-\Delta w= e^w\quad \hbox{in}\quad \rr^2,\qquad
\int\limits_{\rr^2}  e^{w(x)}dx<+\infty.
\end{equation}

 Let us introduce the projection  $P  u$ of  a function $u$
into $H^1_0(\Omega),$ i.e.
$$
 \Delta P u=\Delta u\quad \hbox{in}\ \Omega,\qquad  P u=0\quad \hbox{on}\ \partial\Omega.
$$
It is well known that
\begin{equation}\label{exp}
  P w _{\de,\xi}(x)=w _{\de,\xi}(x)-\log 8\de^2+8\pi H(x,\xi)+ O\(\delta^2\)\
  \end{equation}
 $C^1-$uniformly with respect to $x\in\overline\Omega$ and $\xi$ in compact sets of $\Omega.$

Let $k\ge1$ be a fixed integer. Given $\bxi \in \cal{F}_k \Omega$, we look for a  solution
to \eqref{s} or equivalently \eqref{ps} as
\begin{equation}\label{ans}
{     u}:={     W}+    \phi,\  {     W}:=( W_1,W_2),\ {W_1} :=\sum\limits_{i=1}^kP w_{\de_i,\xi_i} -{1\over2}z(\cdot,\bxi)   ,\ {W_2} :=-{1\over2} \sum\limits_{i=1}^kP w_{\de_i,\xi_i} +z(\cdot,\bxi)  ,
\end{equation}
where  the function $z _{\bxi} $ solves equation \eqref{ome}, the concentration parameter $\de_i$ satisfies
\begin{equation}\label{de}
4\de_i^2=\la  d_i(\bxi)\ \hbox{with}\ d_i(\bxi):=\exp\[8\pi H(\xi_i,\xi_i)+\sum\limits_{j=1\atop j\not=i}^k8\pi G(\xi_j,\xi_i)-{1\over2} z (\xi_i,\bxi)\].\end{equation}
In the following, we agree that $w_i:=w_{\de_i,\xi_i}.$
It is well known that all solutions $\psi \in \mathrm{H}$ (see \eqref{hjs} for the definition of $\mathrm{H}$) of
$$
-\Delta \psi= e^{w_{\delta,\xi}}\psi \quad \hbox{in}\quad \rr^2$$
are linear combinations of the functions
$$Z^0_{\delta,\xi}(x)={\delta^2-|x-\xi|^2\over \delta^2+|x-\xi|^2},\ Z^1_{\delta,\xi}(x)={x_1-\xi_1\over \delta^2+|x-\xi|^2},\ Z^2_{\delta,\xi}(x)={x_2-\xi_2\over \delta^2+|x-\xi|^2}.$$
We introduce their projections $PZ^j_{\delta,\xi}$ onto $H^1_0(\Omega).$ It is well known that
 \begin{equation}\label{pz0}
PZ^0_{\delta,\xi}(x)=Z^0_{\delta,\xi}(x)+1+ O\(\de^2\)={ 2 \de^2 \over \de^2+|x-\xi|^{2} }+ O\(\delta^2\)
\end{equation}
 and
  \begin{equation}\label{pzi}
PZ^j_{\delta,\xi}(x)=Z^j _{\delta,\xi}(x) + O\(1\)={x_i-\xi_i\over \delta^2+|x-\xi|^2} + O\(1\),\ j=1,2
\end{equation}
$C^1-$uniformly with respect to $x\in\overline\Omega$ and $\xi$ in compact sets of $\Omega.$
\\
We agree that $Z^j_i:=Z^j_{\delta_i,\xi_i}$ for any $j=0,1,2$ and $i=1,\dots,k.$
set
$$      {K }:=\hbox{span}\left\{PZ^1_{i},\ PZ^2_{i},\ i=1,\dots,k\right\}\times\{0\} $$
and
$$      {K ^\perp }:= \left\{\phi_1\in H^1_0(\Omega)\ :\ \int\limits_\Omega \nabla \phi_1(x) \nabla PZ^j_{j}(x)dx= 0,\ j=1,2,\ i=1,\dots,k\right\}\times H^1_0(\Omega). $$
We remark that we do not require any orthogonality condition on the second component $\phi_2.$\\
We also denote by
$$     \Pi : H^1_0(\Omega)\times H^1_0(\Omega)\to       {K },\      {\Pi ^\perp}: H^1_0(\Omega)\times H^1_0(\Omega)\to       {K ^\perp }$$
the corresponding projections.
To solve problem \eqref{ps} we will solve the couple of equations:
\begin{equation}\label{equ1}
     {\Pi ^\perp}\[      u-     {i^*_p}\(      F(      u)\)\]=0
\end{equation}
and
\begin{equation}\label{equ2}
     {\Pi  }\[      u-     {i^*_p}\(      F(      u)\)\]=0
\end{equation}
\\

\begin{lemma}\label{stime-utili} For any ${\cal C} \subset \mathcal{F}_k\Omega$ compact, and $\bxi \in {\cal C}$, there holds:
\begin{equation}\label{der1}
\| Pw_i\|=O\(|\log\la|^{1/2}\)\ \hbox{and}\ \|\nabla_{\bxi} Pw_i\|=O\({1\over\sqrt\la}\),
\end{equation}
which implies
\begin{equation}\label{der2}
\| W\|=O\(|\log\la|^{1/2}\)\ \hbox{and}\ \|\nabla_{\bxi} W\|=O\({1\over\sqrt\la}\).
\end{equation}
Moreover, for any $i=1,\dots,k$ and $j=1,2$, there exists $a>0$ so that
 \begin{equation}\label{der3}
\| PZ_i^j\|=\frac{a}{\sqrt\la}+o\({1\over \sqrt\la}\)\ \hbox{and}\  \langle PZ_i^j, PZ_\ell^k\rangle=o\({1\over \la}\)\ \hbox{if $i\not=\ell$ or $j\not=k$.}
\end{equation}
Finally,
\begin{equation}\label{der4}
 \|\nabla_{\bxi} PZ_i^j\|=O\({1\over \la}\).
\end{equation}
\end{lemma}

\begin{proof}
It follows from direct computation by taking into account \eqref{exp}, \eqref{de} and \eqref{pzi}.
\end{proof}

\section{The finite dimensional reduction}

\subsection{Estimate of the error term}
 The next proposition provides an estimate of the error
up to which the couple $(W_1,W_2)$ solves the system \eqref{s}.

First of all, we perform the following estimate.
\begin{lemma}\label{stima-E} Let $\mathcal C \subset \mathcal F_k \Omega $ be a fixed compact set, and $\bxi \in \mathcal{C}$. Define:

$$E(x):=\sum\limits_{i=1}^k e^{ w_i(x)}-2\la e^{  W_1(x)}, \ \ \ E_0(x):= \Delta  z (x,\bxi) + 2\rho_2{e^{W_2(x)}\over \int\limits_\Omega e^{W_2(x)}dx}.$$

For any $p\geq 1$ there exists $\lambda_0>0$ and $C>0$ such that for any $\la \in(0, \la_0)$, $\bxi\in\mathcal C$,

$$
\|E\|_p\le C \lambda^{2-p\over 2p}\ , \ \ \|\nabla_{\bxi} E\|_p\le C \lambda^{1-p\over p},$$
$$
\|E_0\|_\infty\le  C \lambda \ , \ \  \|\nabla_{\bxi} E_0\|_\infty \le C \lambda .$$

\end{lemma}
\begin{proof}
Let us estimate $E$. Let $\eta>0$ be such that $|\xi_i-\xi_j|\ge 2\eta$ and $d(\xi_i,\partial\Omega)\ge2\eta.$
Then we have
\begin{align}\label{er1}
\|E\|_{p}\le &\sum\limits_{i=1}^k\|e^{ w_i(x)}-2\la e^{  W_1(x)}\|_{L^p(B(\xi_i,\eta))}+\sum\limits_{i,j=1\atop i\not=j}^k\|e^{ w_j(x)} \|_{L^p(B(\xi_i,\eta))}\nonumber\\ & +
\Big\|\sum\limits_{i=1}^k e^{ w_i(x)}-2\la e^{  W_1(x)}\Big\|_{L^p(\Omega\setminus \cup_i B(\xi_i,\eta))}
\end{align}
Let us estimate the first term in \eqref{er1}, which is the leading term.
First of all, we point out that
\begin{equation}\label{w1}
W_1(x)=\sum\limits_{i=1}^k\log {1\over \(\delta_i^2+|x-\xi_i|^2\)^2}+8\pi H(x,\xi_i)-{1\over2}z(x,\bxi)+ O(\la)\quad \hbox{for any}\ x\in\Omega.
\end{equation}
As a consequence, for any $i=1,\dots,k$ we immediately get
\begin{equation}\label{ew1}
e^{W_1(\delta_i y+\xi_i)}= { d_i(\bxi)\over  8\delta_i^4}e^{w(y)}\[1+\sum\limits_{i=1}^k O(\delta_i |y|)+O(\la)\]\quad \hbox{for any}\ y\in   B\Big(0,\frac{\eta}{\de_i}\Big),
\end{equation}
where $d_i(\bxi)$ is defined in \eqref{de} and $w$ is defined in \eqref{w}.
Then we scale $x=\delta_i y+\xi_i$ and we get
\begin{align*}
e^{ w_i(x)}-2\la e^{  W_1(x)}&={8 \over\de_i^2\(1+|y|^2\)^2}\[\underbrace{\(1- {\la\over 4\de_i^2} d_i(\bxi) \)}_{=0\ \hbox{by}\ \eqref{de}}+\sum_iO(\de_i|y|)+O(\la)\]
\\ &= O\({1\over \sqrt\la}{|y| \over   \(1+|y|^2\)^2}\)+O\({1 \over   \(1+|y|^2\)^2}\)\end{align*}
Therefore, it follows that
\begin{equation}\label{ei}
\|e^{ w_i(x)}-2\la e^{  W_1(x)}\|_{L^p(B(\xi_i,\eta))}= O\(\la^{2-p\over 2p}\)\ \hbox{for any}\ p\geq 1.
\end{equation}
Now,   by \eqref{de} and \eqref{w1} we immediately deduce that
\begin{equation}\label{eij}
 \|e^{ w_j(x)} \|_{L^\infty(B(\xi_i,\eta))} = O\(\la \)\ \hbox{if}\ i\not=j\end{equation}
 and
 \begin{equation}\label{eout}
 \|\sum\limits_{i=1}^k e^{ w_i(x)}-2\la e^{  W_1(x)}\|_{L^\infty(\Omega\setminus \cup_i B(\xi_i,\eta))}=O(\la).
\end{equation}
 Therefore, by \eqref{er1}, \eqref{ei}, \eqref{eij} and \eqref{eout} we deduce
 \begin{align}\label{E}
\|E\|_{p}=  O\(\la^{2-p\over 2p}\) .
\end{align}

\medskip
Let us estimate $E_0.$
First of all, we point out that
\begin{align}\label{w2}
W_2(x)&=\sum\limits_{i=1}^k\log {  \(\delta_i^2+|x-\xi_i|^2\) }-4\pi H(x,\xi_i)+z(x,\bxi)+O(\la)\nonumber\\ &=\log h(x,\bxi)+z(x,\bxi)+O(\la)
\quad \hbox{for any}\ x\in\Omega.
\end{align}
As a consequence, we immediately get
\begin{equation}\label{ew2}
e^{W_2(x)}= h(x,\bxi)e^{z(x,\bxi)} +O(\la)\quad \hbox{for any}\ x\in\Omega
\end{equation}
where $h$ is defined in \eqref{h}. Therefore,
\begin{align*}
{1\over2}E_0(x)&={1\over2}\Delta  z  +\rho{\Pi_i\(\de_i^2+|x-\xi_i|^2\)  e^{ -4\pi H(x,\xi_i)+z(x,\bxi)+O(\la)} \over  \int\limits_\Omega h(x,\xi)e^{z(x,\bxi)}dx +O(\la) }\\ &=\underbrace{ {1\over2}\Delta  z  +\rho{ h(x,\bxi)  e^{  z(x,\bxi)  }\over  \int\limits_\Omega h(x,\bxi)e^{z(x,\bxi)}dx} }_{=0\ \hbox{by}\ \eqref{ome}} +O(\la) \\ &=O(\la),
\end{align*}
which implies
\begin{equation}\label{e0}
\|E_0\|_{\infty}= O\(\la\).
\end{equation}
\medskip
Now let us consider the estimates of the derivatives of $E$. A straightforward computation shows that
\begin{align*}
\partial_{\xi_i^j} E&=\sum_\ell e^{w_l} \partial_{\xi_i^j} w_l-2\la e^{W_1}\(\sum_\kappa \partial_{\xi_i^j} Pw_\kappa-{1\over2}\partial_{\xi_i^j} z(\cdot,\bxi)\)\\
& =\sum_{\ell,\kappa\atop \ell\not=\kappa} e^{w_l} \partial_{\xi_i^j}P w_\kappa+\sum_\ell  e^{w_l}     \partial_{\xi_i^j}( Pw_\ell-w_\ell)\\
&+\(\sum_\ell  e^{w_l}-2\la e^{W_1}\)  \sum_\kappa \partial_{\xi_i^j} Pw_\kappa\\
&  +\la e^{W_1} \partial_{\xi_i^j} z(\cdot,\bxi)\\
\end{align*}
and the claim easily follows by Lemma \ref{zeta}, estimates \eqref{exp} and \eqref{E} and Remark \ref{stime-utili}.\\
The proof of the estimate of the derivatives of $E_0$ can be carried out in a similar way. More precisely, we have
\begin{align*}{1\over2\rho_2}\partial_{\xi_i^j} E_0&=-
{\(\partial_{\xi_i^j}z(\cdot,\bxi)h(\cdot,\bxi)+\partial_{\xi_i^j} h(\cdot,\bxi)\) e^{z(\cdot,\bxi)}\over\int\limits_\Omega h(x,\bxi)e^{z(x,\bxi)}dx}\\ &+h(\cdot,\bxi)e^{z(\cdot,\bxi)}{\int\limits_\Omega\(\partial_{\xi_i^j}z(x,\bxi)h(x,\bxi)+\partial_{\xi_i^j} h(x,\bxi)\) e^{z(x,\bxi)}dx\over\(\int\limits_\Omega h(x,\bxi)e^{z(x,\bxi)}dx\)^2}
\\ &+{e^{W_2}\partial_{\xi_i^j} W_2\over\int\limits_\Omega e^{W_2(x)}dx}-e^{w_2}{\int\limits_\Omega e^{W_2(x)}\partial_{\xi_i^j} W_2(x)dx\over\(\int\limits_\Omega e^{W_2(x)}dx\)^2}
\end{align*}
 and the claim follows by Lemma \ref{zeta}, estimate \eqref{w2} just taking into account that estimate \eqref{w2} holds also for the derivatives, namely
\begin{equation}\label{w2-der}
\partial_{\xi_i^j}W_2(x)={\partial_{\xi_i^j} h(x,\bxi)\over  h(x,\bxi)} +\partial_{\xi_i^j}z(x,\bxi)+O(\la)
\quad \hbox{for any}\ x\in\Omega.
\end{equation}
\end{proof}

We are now in position to estimate the error term, namely,
\begin{equation}\label{erre}
      R:=     {\Pi^\perp}\(     {i^*_p}\(      F (      W ) \)-      W\).
\end{equation}

Observe that by \eqref{isp} we get
$$\| {     R}\|\le C \| \tilde{R}\|_{p},$$
with
\begin{equation} \label{Rtilde} \tilde{R}= -\Delta W + F(W). \end{equation}

Next lemma is devoted to the estimate of the above term.

\begin{lemma}\label{error}
  Let $\mathcal C \subset \mathcal F_k \Omega $ be a fixed compact set. For any $p\in(1,2)$ there exists $\lambda_0>0$ and $C>0$ such that for any $\la \in(0, \la_0)$, for any $\bxi\in\mathcal C$
it holds
$$
\| {     \tilde{R}}\|\le C \lambda^{2-p\over 2p}, \ \ \ \| \nabla_{\bxi}  \tilde{R} \| \le C \lambda^{1-p\over p}.
$$
As a consequence, for any fixed $\e>0$,

$$\|R\|\leq C\la^{\frac12-\e}$$
\end{lemma}

\begin{proof}
Observe that $\tilde{R}=(\tilde{R}_1, \ \tilde{R}_2)$, where

\begin{align*} \tilde{R}_1(x):=-\Delta W_1(x)-2\la e^{ W_1(x)} +\rho\frac{e^{ W_2(x)}}{\int\limits_\Omega   e^{W_2(x)}dx}, \\
\tilde{R}_2(x):=-\Delta W_2(x)-2\rho_2\frac{e^{ W_2(x)}}{\int\limits_\Omega  e^{W_2(x)}dx} +\la  e^{ W_1(x)} .\end{align*}
We remark that
\begin{align*}
\tilde{R}_1(x)&=-\Delta \(\sum_iP w_i(x) -{1\over2}z(x,\bxi)\)-2\la e^{ W_1(x)} +\rho{e^{W_2(x)}\over \int\limits_\Omega e^{W_2(x)}dx}\\
&=\underbrace{\sum_ie^{ w_i(x)}-2\la e^{  W_1(x)}}_{E (x)} +\underbrace{{1\over2}\Delta  z (x,\bxi) + \rho{e^{W_2(x)}\over \int\limits_\Omega e^{W_2(x)}dx}}_{\frac12 E_0(x)} .\\
\end{align*}
and
\begin{align*}
\tilde{R}_2(x)&=-\Delta \(-{1\over2}\sum_iP w_i(x) +z(x,\xi) \) -2\rho_2{e^{W_2(x)}\over \int\limits_\Omega e^{W_2(x)}dx})+\la e^{ W_1(x)}\\
&=-\underbrace{\sum_i\frac12 e^{ w_i(x)}-\la e^{  W_1(x)}}_{\frac12 E (x)} -\underbrace{\Delta  z (x,\bxi) +2 \rho{e^{W_2(x)}\over \int\limits_\Omega e^{W_2(x)}dx}}_{E_0(x)} .\\
\end{align*}
We now conclude by Lemma \ref{stima-E}.

\end{proof}

\subsection{Analysis of the linearized operator}

Let us consider the following linear problem:  given $\xi\in\Omega$ and  $ h_1,h_2\in H^1_0(\Omega)$, find a function
$\phi$ and constants $c_1,\, c_2$ satisfying
\begin{equation}\label{lla0}\left\{\begin{aligned}&
-\Delta \phi_1  - \sum\limits_{i=1}^k e^{w_i}\phi_1   +\rho\[{e^{W_2}\phi_2\over \int\limits_\Omega e^{W_2}dx}-
{e^{W_2} \int\limits_\Omega e^{W_2} \phi_2dx\over\( \int\limits_\Omega e^{W_2}dx \)^2}  \]=-\Delta h_1
{+\sum\limits_{j=1,2\atop i=1,\dots,k} c_{ij}Z^j_{i} e^{w_{i}}},\\
 &-\Delta \phi_2  - 2\rho_2\[{e^{W_2}\phi_2\over \int\limits_\Omega e^{W_2}dx}-
{e^{W_2} \int\limits_\Omega e^{W_2} \phi_2dx\over\( \int\limits_\Omega e^{W_2}dx \)^2}  \]   +{1\over2}\sum\limits_{i=1}^k e^{w_i}\phi_1=-\Delta h_2,\\ &
\phi_1,\phi_2\in H^1_0(\Omega),\qquad
\int\limits_\Omega \nabla \phi_1 \nabla PZ^j_{j}dx= 0,\ j=1,2,\ i=1,\dots,k.
\end{aligned}\right.
\end{equation}
In order to solve problem \eqref{lla0}, we need  to establish an a priori estimate. We  first consider an intermediate problem.
\begin{lemma}\label{inv}Let $\mathcal C \subset \mathcal F_k \Omega $ be a fixed compact set. For any $p>1$ there exists $\lambda_0>0$ and $c>0$ such that for any $\la \in(0, \la_0)$, for any $\bxi\in\mathcal C$
and
for any solution
$(\phi_1,\phi_2) $ of
\begin{equation}\label{lla}\left\{\begin{aligned}&
-\Delta \phi_1  - \sum\limits_{h=1}^k e^{w_h}\phi_1   +\rho\[{e^{W_2}\phi_2\over \int\limits_\Omega e^{W_2}dx}-
{e^{W_2} \int\limits_\Omega e^{W_2} \phi_2dx\over\( \int\limits_\Omega e^{W_2}dx \)^2}  \]=-\Delta h_1
,\\
 &-\Delta \phi_2  - 2\rho\[{e^{W_2}\phi_2\over \int\limits_\Omega e^{W_2}dx}-
{e^{W_2} \int\limits_\Omega e^{W_2} \phi_2dx\over\( \int\limits_\Omega e^{W_2}dx \)^2}  \]   +\frac12\sum\limits_{h=1}^k e^{w_h}\phi_1=-\Delta h_2,\\ &
\phi_1,\phi_2\in H^1_0(\Omega),\qquad
\int\limits_\Omega \nabla \phi_1 \nabla PZ^j_{i}dx= 0,\ j=1,2,\ i=1,\dots,k,
\end{aligned}\right.
\end{equation}
 the following holds  $$\|\phi_1\|+\|\phi_2\| \leq c\Big( |\log\la|  \|h_1\| +  \|h_2\| \Big).$$

\end{lemma}
\begin{proof}
We argue by contradiction. Assume that there exist $p>1,$ sequences $\la:=\la_n\to0,$ $ \bxi:= \bxi_n\to\bxi^*:=(\xi_1^*,\dots,\xi_k^*) \in\mathcal F_k \Omega ,$ $h_\ell:={h_\ell}_n\in H^1_0(\Omega)$ and $\phi_\ell:={\phi_\ell}_n\in \mathrm{W}^{2, 2}(\Omega)$ for $\ell=1,2$,
which solve \eqref{lla}
 and
\begin{equation}\label{inv2}
\|\phi_1\|+\|\phi_2\| =1 \end{equation}
and
\begin{equation}\label{inv3}
   |\log\la |  \|h_1\| +\|h_2\| \to0.\end{equation}
  For any $i=1,\dots,k,$
we define the functions
 $  \tilde \phi_1^i(y):=\phi_1 \(\de _iy+\xi_i\) $ if $y\in \tilde \Omega_i :={\Omega-\xi_i\over \de_i } $ and  $  \tilde \phi_1^i(y):=0$ if $y\in \rr^2\setminus\tilde \Omega_i . $
\\
{\em Step 1: we will show that for any $i=1,\dots,k$
\begin{equation}\label{step1.1}  \tilde \phi_1^i (y)\to \gamma_i  {1-|y|^{2}\over1+|y|^{2}}\ \hbox{  weakly in $\mathrm{H} (\rr^2)$ and strongly in $\mathrm{L} (\rr^2) $,\  for some $ \gamma_i \in\rr $} \end{equation}
where  $\mathrm{H} (\rr^2)$ and $\mathrm{L} (\rr^2) $ are defined in \eqref{ljs}, \eqref{hjs}, and
\begin{equation}\label{step1.2}  \phi_2(x)\to 0 \ \hbox{weakly in $H^1_0(\Omega)$ and strongly in $L^q(\Omega)$ for any $q\ge2.$}\end{equation}
  }

Let us prove \eqref{step1.2}.
Let $\psi_2\in C^\infty_c(\Omega\setminus\{\xi_1^*,\dots,\xi_k^*\}).$ We multiply the second equation in \eqref{lla} by $\psi_2,$   we integrate over $\Omega$ and we get
\begin{align}\label{1.1.1}&\int\limits_\Omega \nabla\phi_2\nabla\psi_2dx  - 2\rho_2\[{\int\limits_\Omega e^{W_2}\phi_2\psi_2dx \over \int\limits_\Omega e^{W_2}dx}-
{\int\limits_\Omega e^{W_2}\psi_2dx \int\limits_\Omega e^{W_2} \phi_2dx\over\( \int\limits_\Omega e^{W_2}dx \)^2}  \]\nonumber \\ &   +\frac12 \sum\limits_h \int\limits_\Omega e^{w_h}\phi_1\psi_2dx=\int\limits_\Omega \nabla h_2\nabla \psi_2dx\end{align}
By \eqref{inv2} we deduce that $\phi_1\to\phi_1^*$ and $\phi_2\to\phi_2^*$  weakly in $H^1_0(\Omega)$ and strongly in $L^q(\Omega)$ for any $q\ge2.$
Then, we use \eqref{inv3}, \eqref{w1} and \eqref{ew2} and passing to the limit in \eqref{1.1.1} we immediately get
$$\int\limits_\Omega \nabla\phi^*_2\nabla\psi_2dx  - 2\rho_2\[{\int\limits_\Omega h(x,\bxi)e^{z(x,\bxi)}\phi^*_2\psi_2dx \over \int\limits_\Omega h(x,\bxi)e^{z(x,\bxi)}dx}-
{\int\limits_\Omega h(x,\bxi)e^{z(x,\bxi)}\psi_2dx \int\limits_\Omega e^{W_2} \phi^*_2dx\over\( \int\limits_\Omega e^{W_2}dx \)^2}  \]=0.$$
Since $\|\phi_2^*\|_{H^1_0(\Omega)}\le 1$ we can conclude that $\phi^*_2\in H^1_0(\Omega)$ solves the problem
$$ -\Delta\phi^*_2  - 2\rho_2\[ { h(x,\xi)e^{z(x,\bxi)}\phi^*_2  \over \int\limits_\Omega h(x,\bxi)e^{z(x,\bxi)}dx}-
{  h(x,\bxi)e^{z(x,\bxi)}  \int\limits_\Omega e^{W_2} \phi^*_2dx\over\( \int\limits_\Omega e^{W_2}dx \)^2}  \]=0\ \hbox{in}\ \Omega$$
and by \eqref{H} we get $\phi^*_2=0$. That proves our claim\\

Let us prove \eqref{step1.1}.
First of all we claim that  each $\tilde \phi_1^i$ is bounded in the space $\mathrm{H} (\rr^2)$. It is immediate to check that
$$\int\limits_{\rr^2}|\nabla \tilde\phi_1^i|^2dy=\int\limits_{\Omega}|\nabla  \phi_1 |^2dx\le1.$$
Next, we multiply  the first equation in  \eqref{lla} by $\phi_1 ,$   we integrate over $\Omega$ and we get
\begin{align*}&\int\limits_\Omega |\nabla\phi_1 |^2dx -\sum\limits_h \int\limits_\Omega e^{w_h(x)} \phi_1  ^2dx  +\rho\[{\int\limits_\Omega e^{W_2}\phi_1\phi_2dx \over \int\limits_\Omega e^{W_2}dx}-
{\int\limits_\Omega e^{W_2}\phi_1dx \int\limits_\Omega e^{W_2} \phi_2dx\over\( \int\limits_\Omega e^{W_2}dx \)^2}  \]\nonumber \\ &   =\int\limits_\Omega  \nabla h_1 \nabla \phi_1dx\end{align*}
which implies for any $i$
\begin{align*}&\int\limits_{\tilde\Omega_i} e^{w (y)} \(\tilde\phi_1^i(y)\)  ^2dy=\int\limits_\Omega e^{w_i(x)} \phi_1  ^2dx\\ &\le
\int\limits_\Omega |\nabla\phi_1 |^2dx   +\rho\[{\int\limits_\Omega e^{W_2}\phi_1\phi_2dx \over \int\limits_\Omega e^{W_2}dx}-
{\int\limits_\Omega e^{W_2}\phi_1dx \int\limits_\Omega e^{W_2} \phi_2dx\over\( \int\limits_\Omega e^{W_2}dx \)^2}  \]-\int\limits_\Omega \nabla h_1\nabla \phi_1dx.\end{align*}

Then, by \eqref{step1.2}, \eqref{inv2}, \eqref{inv3},    \eqref{ew2} and \eqref{ei},    we immediately get
 $$\int\limits_{\rr^2} {1\over 1+|y|^2} (\tilde\phi_1^i)^2(y)  dy \le c$$
 for some positive constant $c.$
 \\ Therefore, by Proposition \ref{compact} we deduce that
 $\tilde \phi_1 ^i \to \tilde\phi_0^i$  weakly in $\mathrm{H} (\rr^2)$ and strongly in $\mathrm{L} (\rr^2) $.
 Now, let $\tilde \psi_i\in C^\infty_c(\rr^2 ).$ Set $\psi_i(x):=\tilde\psi_i\({x-\xi_i\over\de_i}\),$ $x\in\Omega.$ We multiply  the first  equation in \eqref{lla} by $\psi_i,$   we integrate over $\Omega$ and we get
\begin{align}\label{1.1.2}&\int\limits_\Omega \nabla\phi_1 \nabla\psi_idx -\ \int\limits_\Omega e^{w_i}\phi_1 \psi_idx-\sum_{h\not=i } \int\limits_\Omega e^{w_h}\phi_1 \psi_idx\nonumber \\ & +\rho\[{\int\limits_\Omega e^{W_2} \phi_2\psi_idx \over \int\limits_\Omega e^{W_2}dx}-
{\int\limits_\Omega e^{W_2} \psi_idx \int\limits_\Omega e^{W_2} \phi_2dx\over\( \int\limits_\Omega e^{W_2}dx \)^2}  \]     =\int\limits_\Omega  \nabla h_1 \nabla \psi_idx.\end{align}
Now we have that
$$\int\limits_\Omega e^{w_h}\phi_1 \psi_idx=O(\de_h^2),$$
because $\psi_i(x)=0$ if $|x-\xi_i|\ge R\delta_i$ for some $R$ and so $|x-\xi_h|\ge |\xi_h-\xi_i|-|x-\xi_i|\ge r$ for some $r.$
Therefore, by \eqref{step1.2}, \eqref{ew2}, by \eqref{inv3},
 we pass to the limit in \eqref{1.1.2} and we get
$$\int\limits_{\rr^2} \nabla\tilde \phi_0^i\nabla\tilde\psi_idy -  \int\limits_{\rr^2}  e^{w(y)}\tilde\phi_0^i\tilde\psi_idy=0,$$
namely $\tilde \phi_0^i\in H(\rr^2)$ solves the linear problem \eqref{plim}.
On the other hand, by the last equations in \eqref{lla} we deduce
$$
\int\limits_{\rr^2}  \tilde\phi_0^i(y) e^{w(y)} {y_j\over 1+|y|^2}dy=0,\ j=1,2. $$
and so the claim follows.\\

{\em Step 2: we will show that   $\gamma_i=0$
for any $i=1,\dots,k.$  }\\

We multiply  the first equation in \eqref{lla} by $   PZ^0 _i$  (see \eqref{pz0}), we integrate over $\Omega$ and we get
\begin{align}\label{1.2}& \int\limits_\Omega \nabla\phi_1\nabla PZ^0_idx -\sum_h\int\limits_\Omega e^{w_h}\phi_1  PZ^0_idx +\rho\[{\int\limits_\Omega e^{W_2} \phi_2PZ^0_idx \over \int\limits_\Omega e^{W_2}dx}-
{\int\limits_\Omega e^{W_2} PZ^0dx \int\limits_\Omega e^{W_2} \phi_2dx\over\( \int\limits_\Omega e^{W_2}dx \)^2}  \]\nonumber \\ &    =\int\limits_\Omega \nabla  h_1 \nabla PZ^0_idx.\end{align}

Now, by \eqref{step1.2}   we get
$$\int\limits_\Omega \nabla\phi_1 \nabla PZ^0_idx =
\int\limits_{\tilde \Omega_i}   e^{w(y)} \tilde\phi_1^i(y) Z^0 (y)dy  $$
where $Z^0(y):={1-|y|^2\over 1+|y|^2}$
and using also \eqref{pz0}   we get
\begin{align*}
\sum_h\int\limits_\Omega e^{w_h}\phi_1 PZ^0_idx &=\int\limits_\Omega e^{w_i}\phi_1  PZ^0_idx+\sum_{h\not=i}\int\limits_\Omega e^{w_h}\phi_1  PZ^0_idx\nonumber\\ &=\int\limits_{\tilde\Omega_i}  e^{w(y)}\tilde\phi_1^i(y)  \[Z^0(y)+ 1+O(\de^2)\]dy\nonumber\\ &+\sum_{h\not=i}O\(\| e^{w_h}\|_{p} \|\phi_1 \|_{H^1_0(\Omega)}
\|PZ^0_i\|_{\infty}\)\\
&=\int\limits_{\tilde\Omega_i}  e^{w(y)}\tilde\phi_1^i(y)  \[Z^0(y)+ 1 \]dy +O\(\la^{1\over p}\).\end{align*}
Moreover, by \eqref{w2}, \eqref{ew2},  \eqref{step1.2} and \eqref{pz0}, we deduce
$$\rho\[{\int\limits_\Omega e^{W_2} \phi_2PZ^0_idx \over \int\limits_\Omega e^{W_2}dx}-
{\int\limits_\Omega e^{W_2} PZ^0_idx \int\limits_\Omega e^{W_2} \phi_2dx\over\( \int\limits_\Omega e^{W_2}dx \)^2}  \]=O\(\la\)$$
and
$$\int\limits_\Omega  \nabla h_1 \nabla PZ^0_idx=O\(\|h_1\|\|PZ^0_i\|\) =O\(\|h_1\| \) .$$
Therefore, by \eqref{1.2} we immediately deduce
\begin{equation}\label{g0}
\lim\limits_{\la\to0}(\log\la)\int\limits_{\tilde\Omega_i} e^{w(y)} \tilde\phi_1^i(y) dy=0.
\end{equation}
\\
Next, we   multiply   the first equation in  \eqref{lla} by $   Pw_{i} $  (see \eqref{exp}), we integrate over $\Omega$ and we get
\begin{align}\label{1.3}& \int\limits_\Omega \nabla\phi_1\nabla Pw_{i}dx -\sum_h \int\limits_\Omega e^{w_h}\phi_1 Pw_{i}dx +\rho\[{\int\limits_\Omega e^{W_2} \phi_2Pw_{i}dx \over \int\limits_\Omega e^{W_2}dx}-
{\int\limits_\Omega e^{W_2} Pw_{i}dx \int\limits_\Omega e^{W_2} \phi_2dx\over\( \int\limits_\Omega e^{W_2}dx \)^2}  \]\nonumber \\ &    =\int\limits_\Omega \nabla  h_1 \nabla Pw_{i}dx.\end{align}
We have
$$\int\limits_\Omega \nabla\phi_1 \nabla Pw_{i}dx=\int\limits_\Omega   e^{w_{i}(x)}\phi_1 (x)dx=\int\limits_{\tilde\Omega_i} e^{w(y)} \tilde \phi_1^i(y)dy=o(1)$$
because of \eqref{step1.1} and the fact that
\begin{equation}\label{zw}
\int\limits_{\rr^2}e^{w(y)}{1-|y|^2\over 1+|y|^2}dy=0.
\end{equation}
Moreover, we have
\begin{align*}
 &\sum_h \int\limits_\Omega e^{w_h}\phi_1  Pw_idx  = \int\limits_\Omega e^{w_i}\phi_1 (x) Pw_idx+\sum_{h\not=i} \int\limits_\Omega e^{w_h}\phi_1  Pw_idx \ \hbox{(we use \eqref{exp} and \eqref{ei})}
 \\ &=\int\limits_{\tilde\Omega_i} e^{w (y)}\tilde\phi_1^i(y) \(-4\log\de_i-2\log(1+|y|^2)+8\pi H(\xi_i,\xi_i)+O(\de_i|y|)+O(\de_i^2)\)dy\\ & \;\;\;\;+\sum_h\int\limits_{\tilde\Omega_h} e^{w (y)}\tilde\phi_1^h(y) \( 8\pi G(\xi_i,\xi_h)+O(\de_h|y|)+O(\de_h^2)\)dy\\ &\ \;\;\;\;\hbox{(we use \eqref{de}, \eqref{g0}, \eqref{step1.2},   \eqref{zw} and \eqref{step1.1})}\\ &=  \gamma_i\int\limits_{\rr^2} e^{w (y)}{1-|y|^2\over 1+|y|^2} \(-2\log(1+|y|^2)\)dy+o(1)
 \end{align*}
Finally, by \eqref{w2}, \eqref{ew2},  \eqref{step1.2} and \eqref{exp}, we get
$$\rho\[{\int\limits_\Omega e^{W_2} \phi_2Pw_{\delta,\xi}dx \over \int\limits_\Omega e^{W_2}dx}-
{\int\limits_\Omega e^{W_2} Pw_{\delta,\xi}dx \int\limits_\Omega e^{W_2} \phi_2dx\over\( \int\limits_\Omega e^{W_2}dx \)^2}  \]=o(1) $$
and by \eqref{inv3} and Lemma \ref{stime-utili}
we get
$$\int\limits_\Omega \nabla  h_1 \nabla Pw_idx=O\(\|h_1\| \|Pw_{\delta,\xi}\| \)=O\(|\log\la|^{1/2}\|h_1\|  \)=o(1).$$
 Therefore, putting all the previous estimates into \eqref{1.3} we get
 $$\gamma_i\int\limits_{\rr^2} e^{w (y)}{1-|y|^2\over 1+|y|^2} \(-2\log(1+|y|^2)\)dy=0,$$
 which immediately gives $\gamma_i=0$ since a straightforward computation shows that
 $$\int\limits_{\rr^2}  {1-|y|^2\over \(1+|y|^2\)^3}  \log(1+|y|^2) dy\not=0.$$
 That concludes the proof of the second step.\\

 {\em Step 3: we will show that a contradiction arises!}\\

We    multiply the first equation in \eqref{lla} by $   \phi_2 $, the second equation in \eqref{lla} by $   \phi_1 $, we subtract the two equations, we integrate over $\Omega$ and we get
\begin{align}\label{1.4}  &
 - \sum_h \int\limits_\Omega e^{w_h}\phi_1 \phi_2dx  + \rho\[{\int\limits_\Omega e^{W_2} \phi_2^2 dx \over \int\limits_\Omega e^{W_2}dx}-
  {\(\int\limits_\Omega e^{W_2} \phi_2 dx\)^2\over\( \int\limits_\Omega e^{W_2}dx \)^2}  \]\nonumber \\ &    +2\rho_2\[{\int\limits_\Omega e^{W_2}\phi_1\phi_2 dx \over \int\limits_\Omega e^{W_2}dx}-
  {\int\limits_\Omega e^{W_2} \phi_1 dx\int\limits_\Omega e^{W_2} \phi_2 dx\over\( \int\limits_\Omega e^{W_2}dx \)^2}  \]  -\frac12\sum_h\int\limits_\Omega e^{w_h}\phi_1^2  dx \nonumber \\ &    =\int\limits_\Omega  \nabla h_1 \nabla \phi_2dx -\int\limits_\Omega \nabla  h_2 \nabla \phi_1dx.\end{align}

 By  \eqref{step1.2}, \eqref{ew2} and \eqref{inv2}, we deduce
\begin{equation}\label{1.5}\rho\[{\int\limits_\Omega e^{W_2} \phi_2^2 dx \over \int\limits_\Omega e^{W_2}dx}-
  {\(\int\limits_\Omega e^{W_2} \phi_2 dx\)^2\over\( \int\limits_\Omega e^{W_2}dx \)^2}  \]=o(1)\end{equation}
  and
\begin{equation}\label{1.6}2\rho_2\[{\int\limits_\Omega e^{W_2}\phi_1\phi_2 dx \over \int\limits_\Omega e^{W_2}dx}-
  {\int\limits_\Omega e^{W_2} \phi_1 dx\int\limits_\Omega e^{W_2} \phi_2 dx\over\( \int\limits_\Omega e^{W_2}dx \)^2}  \]=o(1).\end{equation}
By \eqref{step1.1}, together with \eqref{ei} and \eqref{inv2}, we deduce
\begin{equation}\label{1.7}2 \sum_h \int\limits_\Omega e^{w_h}\phi_1 ^2dx= \sum \int\limits_{\tilde\Omega_h} e^{w(y)}\(\tilde\phi_1^h(y)\) ^2 dy =o(1).\end{equation}
 By \eqref{inv2} and \eqref{inv3} we deduce
\begin{equation}\label{1.8}\int\limits_\Omega  \nabla h_1 \nabla \phi_1dx=o(1)\ \hbox{and}\ \int\limits_\Omega \nabla  h_2 \nabla \phi_2dx=o(1).\end{equation}
Therefore, by \eqref{1.4} we deduce
\begin{equation}\label{1.9}2 \sum_h \int\limits_\Omega e^{w_h}\phi_1 \phi_2dx =o(1).\end{equation}
Finally, we    multiply the first equation in \eqref{lla} by $   \phi_1 $, the second equation in \eqref{lla} by $   \phi_2 $, we sum the two equations, we integrate over $\Omega$ and we get (taking into account \eqref{inv2})
\begin{align*} 1&=\int\limits_\Omega |\nabla\phi_1|^2dx+\int\limits_\Omega |\nabla\phi_2|^2dx \nonumber\\ &
= \sum_h \int\limits_\Omega e^{w_h}\phi_1 ^2dx  +2\rho_2\[{\int\limits_\Omega e^{W_2}\phi_2^2dx \over \int\limits_\Omega e^{W_2}dx}-
  {\(\int\limits_\Omega e^{W_2} \phi_2 dx\)^2\over\( \int\limits_\Omega e^{W_2}dx \)^2}  \]\nonumber \\ & \;\;\;\;   -\frac12\sum_h\int\limits_\Omega e^{w_h}\phi_1\phi_2dx-\rho\[{\int\limits_\Omega e^{W_2}\phi_1\phi_2 dx \over \int\limits_\Omega e^{W_2}dx}-
{\int\limits_\Omega e^{W_2}\phi_1dx \int\limits_\Omega e^{W_2} \phi_2dx\over\( \int\limits_\Omega e^{W_2}dx \)^2}  \]\nonumber \\ &  \;\;\;\;  +\int\limits_\Omega \nabla  h_1 \nabla \phi_1dx +\int\limits_\Omega  \nabla h_2 \nabla \phi_2dx\\ &=o(1),\end{align*}
because of \eqref{1.5}--\eqref{1.9}, and a contradiction arises!

 \end{proof}
Now we are ready to derive an a priori estimate for problem \eqref{lla0}.

\begin{prop}\label{inv1}
Let $\mathcal C \subset \mathcal F_k \Omega $ be a fixed compact set. For any $p>1$ there exists $\lambda_0>0$ and $c>0$ such that for any $\la \in(0, \la_0)$, for any $\bxi\in\mathcal C$
and for any solution
$(\phi_1,\phi_2 ) $ and $c_{ij},$ $j=1,2$, $i=1,\dots,k$ of \eqref{lla0}, the following holds
 $$\|\phi_1\|+\|\phi_2\| \le c  |\log\la|  \(\|h_1\|  +  \|h_2\| \) .$$
\end{prop}
\begin{proof}
For any $q\ge1$ we have
\beq\label{cii}\|PZ_i^je^{w_i}\|_{q}=
O\(\la^{\frac{2-3q}{2q}}\)\ \hbox{and} \ \|PZ_h^je^{w_i}\|_{q}
=O\(\la^{\frac{1- q}{q}}\)\hbox{ for }h\neq i
\eeq
and
$$ \|PZ_i^j\|_{2}=O(1).$$
Lemma \ref{inv}  combined with Lemma \ref{stime-utili} yields
 \begin{equation}\label{combi1}\|\phi_1\| + \|\phi_2\|  =O\( |\log\la| \( \|h_1\| + \|h_2\| +\la^{-\frac12}\sum\limits_{j=1,2\atop i=1,\dots,k} |c_{ij}| \)\).\end{equation}
Hence it suffices to estimate the values of the constants $c_{ij}$.
 We multiply the first equation of \eqref{lla0} by $PZ_i^j$ and, using again Lemma \ref{stime-utili}, we find
\begin{equation}\label{multi}\begin{aligned}& \int_\Omega \phi_1e^{w_i} Z_{i}^jdx-\sum_h\int_\Omega e^{w_h}\phi_1PZ_{i}^j+O\bigg(\int_{\Omega}|\phi_2||PZ_i^j|dx +\int_{\Omega}|\phi_2|\int_\Omega|PZ_i^j|dx \bigg) \\ &=\int_\Omega   \nabla h_1 \nabla  PZ_i^jdx +c_{ij}\int_\Omega  Z_i^jP Z_i^je^{w_i}dx+ \sum\limits_{\ell\not=j\atop \kappa\not= i} o\Big(\frac{|c_{\ell\kappa}|}{\la}\Big)\end{aligned}\end{equation}
Let us fix $q\in (1,+\infty)$ sufficiently close to 1. By using   \eqref{pzi}, the first part in \eqref{multi} can be estimated as
\begin{align*}
&\int_\Omega \phi_1e^{w_{i}}\(Z_i^j-PZ_{i}^j\)dx-\sum_{h\not=i}\int_\Omega e^{w_h}\phi_1PZ_{i}^j \\ &=O\(
\|  e^{w_i}\|_{q)}
\|\phi_1\| \)+\sum_{h\not=i} O\(\|  e^{w_h} PZ_{i}^j\|_{q}
\|\phi_1\| \)\\ &=O\( \la^{1-q\over q}\|\phi_1\|\).\end{align*} Furthermore
$$\int_{\Omega}|\phi_2||PZ_i^j|dx+\int_{\Omega}|\phi_2|dx\int_\Omega|PZ_i^j|dx=O\(\|\phi_2\|\|PZ_i^j\|_{2}\)=O\(\|\phi_2\|\).$$
Next we examine the right hand side and by Lemma \ref{stime-utili} we deduce
$$\int_\Omega \nabla  h_1 \nabla PZ_i^jdx =O\( \|h_1\| \|PZ_i^j\| \)= O\(\|h_1\|_p\la^{-\frac12}\).$$
By inserting the above estimates into \eqref{multi} and recalling \eqref{der2} we get
\begin{equation}\label{combi}|c_{ij}|+o\(\sum\limits_{\ell\not=j\atop k\not= i}|c_{\ell k}|\)=O\( \la^{\frac{1}{q}}\|\phi_1\|+\la^{1\over2 }\|h_1\| +\la \|\phi_2\|\).
\end{equation}
We sum \eqref{combi} over all the indices $j=1,2$ and $i=1,\dots,k$ and we get
\begin{equation}\label{combi2}\sum\limits_{ j=1,2\atop i=1,\dots,k}|c_{ij}|=O\( \la^{\frac{1}{q}}\|\phi_1\|+\la^{1\over2}\|h_1\| +\la \|\phi_2\|\).
\end{equation}
Combining \eqref{combi2} with \eqref{combi1} we get the thesis provided that we choose $q$ sufficiently close to $1$.
\end{proof}

Once that a priori estimate has been obtained, we are in the position to prove the solvability result.
We denote by $      L:      {K^\perp}\to      {K^\perp}$ the linear operator defined by
\begin{equation}\label{elle}
      L(     \phi):=     \phi-     {\Pi^\perp}\(     {i^*_p}\(      M(      W)      \phi \)\), \quad \phi=(\phi_1,\phi_2)
\end{equation}
where
\beq\label{effe}
      M(      W)(     \phi_1,\phi_2):=   \( \begin{aligned}    \sum\limits_{h=1}^k e^{w_h}\phi_1   -\rho\[{e^{W_2}\phi_2\over \int\limits_\Omega e^{W_2}dx}-
{e^{W_2} \int\limits_\Omega e^{W_2} \phi_2dx\over\( \int\limits_\Omega e^{W_2}dx \)^2}\]  \\ - \frac12\sum\limits_{h=1}^k e^{w_h}\phi_1   +2\rho_2\[{e^{W_2}\phi_2\over \int\limits_\Omega e^{W_2}dx}-
{e^{W_2} \int\limits_\Omega e^{W_2} \phi_2dx\over\( \int\limits_\Omega e^{W_2}dx \)^2}  \]
\end{aligned}\)\eeq

Notice  that  problem   \eqref{lla0} can be written in the operator form

\beq\label{fre}\phi-\Pi^{\perp}(i^*_p(M(W) \phi))=\Pi^{\perp}(i^*_p(h)),\qquad \phi=(\phi_1,\phi_2).\eeq
\begin{prop}\label{ex} Let $\mathcal C \subset \mathcal F_k \Omega $ be a fixed compact set. For any $p>1$ there exists $\lambda_0>0$ and $c>0$ such that for any $\la \in(0, \la_0)$, for any $\bxi\in\mathcal C$,  for any $h\in K^\perp$ there is a unique solution to the problem \beq\label{fre1}L(\phi)=h.\eeq
Moreover  $$\|\phi\| \leq C |\log \lambda | \|L(\phi)\|\quad \forall \phi\in K^{\perp}.$$

\end{prop}

\begin{proof} The existence follows from Fredholm's
alternative.
Indeed the operator $$\phi\mapsto \Pi^{\perp}(i^*_p(M(W) \phi))$$ is a compact operator in $K^\perp$.
   Using Fredholm's alternatives,  \eqref{fre} has a unique  solution for each $ h\in K^\perp$, if and only if \eqref{fre}  has a unique solution
for $h=0$. Let   $\phi\in K^\perp$ be a
solution of $\phi-\Pi^{\perp}(i^*_p(M(W) \phi))=0$; then  $\phi$ solves
the system \eqref{lla0} with $h=0$ for some
$c_1,\,c_2\in\R$. Proposition \ref{inv1} implies $\phi\equiv 0.$

Once we have existence, the norm estimates follows directly from Proposition \ref{inv1}.
\end{proof}

\subsection{The nonlinear problem} We recall that our goal is to solve problem \eqref{equ1}.
In what follows we denote by
 $      N(     \phi)$ the nonlinear operator
\begin{equation}\label{enne}
      N(     \phi):=     {\Pi^\perp}\(     {i^*_p}\(      F (      W+     \phi)-      F (      W)-      F' (      W)     \phi \)\).
\end{equation}
and by $S(\phi)$ the linear operator (see \eqref{effe})
\begin{equation}\label{esse}
      S(     \phi):=     {\Pi^\perp}\(     {i^*_p}\( F'(W)\phi-M(W)\phi\)\).
\end{equation}
Therefore, equation \eqref{equ1} turns out to be equivalent to the problem
\beq\label{goal}     L(     \phi)=      N(     \phi)+  S(\phi)   + R,\eeq where $R$ is the error term defined in \eqref{erre}.
We need the following   auxiliary lemmas.
\begin{lemma}\label{aux0}Let $\mathcal C \subset \mathcal F_k \Omega $ be a fixed compact set. For any $p>1$ there exists $\lambda_0>0$ and $c>0$ such that for any $\la \in(0, \la_0)$, for any $\bxi\in\mathcal C$  and any $\phi ,\psi \in H^1_0(\Omega)\times H^1_0(\Omega)  $
$$\|S(\phi)-S(\psi)\|\le c\la^{\frac{2-p}{2p}}\|\phi-\psi\|.$$
\end{lemma}
\begin{proof} A direct computation shows that   $S(\phi)$ reduces to
$$S(\phi):={\Pi^\perp}\(     {i^*_p}\(  \( \la e^{W_1}-\sum\limits_h e^{w_h}\)\phi_1,-\frac12 \( \la e^{W_1}-\sum\limits_h e^{w_h}\)\phi_1\)\) $$
and then the claim  immediately follows by \eqref{E}.
\end{proof}

\begin{lemma}\label{aux}Let $\mathcal C \subset \mathcal F_k \Omega $ be a fixed compact set. For any $p\geq1$ and $q>1$  there exists $\lambda_0>0$ and $c>0$ such that for any $\la \in(0, \la_0)$, for any $\bxi\in\mathcal C$   and any $u\in H^1_0(\Omega)\times H^1_0(\Omega) $ with $\|u\|\le r_0$:

\begin{enumerate}

\item[a)] $\|\la f'(W_1+u_1) \, v\|_p\leq C\la^{\frac{1-pq}{pq}}e^{c r_0^2}\|v\| \quad \forall v \in H^1_0(\Omega),$

\item[b)] $\|g'(W_2+u_2)\, v \|_p \leq C e^{c r_0^2} \|v\| \quad \forall v \in H^1_0(\Omega).$

\item[c)] $\|\la f''(W_1+u_1)\, v\, z\|_p\leq C\la^{\frac{1-pq}{pq}}e^{c r_0^2}\|v\|\|z\|\quad \forall v, \ z\in H^1_0(\Omega),$

\item[d)] $\|g''(W_2+u_2)\, v\, z\|_p\leq Ce^{c r_0^2}\|v\|\|z\|\quad \forall v, \ z\in H^1_0(\Omega).$
\end{enumerate} \end{lemma}

\begin{proof} We give the complete proof for inequalities c), d), the others being easier. We point out that by  H\"older's inequality  with ${1\over q} +{1\over r}+{1\over s}+{1\over t} =1$,

$$\begin{aligned}
\left\| \la f''\(W_1+u_1\)v_1z_1\right\| _p &\le \left\| \la e^{W_1 } \right\| _{pq}\left\|   e^{ u_1} \right\| _{pr} \left\|v_1 \right\| _{ps}\left\| z_1 \right\| _{pt}\nonumber\\ &\hbox{(we use the continuity of $H^1_0(\Omega)\hookrightarrow L^p(\Omega)$)}\nonumber\\ &\le c \left\| \la e^{W_1 } \right\| _{pq}\left\|   e^{ u_1} \right\| _{pr}\left\|v_1 \right\|  \left\| z_1 \right\|\nonumber\\ &\hbox{(we use Lemma \ref{tmt})}\nonumber\\ &\le c \left\| \la e^{W_1 } \right\| _{pq}  e^{ \frac {pr}{16\pi}\|u_1\|^2} \left\|v_1 \right\|  \left\| z_1 \right\|
\nonumber\\ &\hbox{(we use  Lemma \ref{stima-E}
)}\nonumber\\ &\le c\Big( \sum_{i=1}^k\left\| e^{w_i } \right\| _{pq} +\la^{\frac{2-pq}{pq}}\Big) e^{ \frac {pr}{16\pi}\|u_1\|^2} \left\|v_1 \right\|  \left\| z_1 \right\|
\nonumber\\ &\le c \la^{1-pq\over pq}  e^{ \frac {pr}{16\pi}\|u_1\|^2} \left\|v_1 \right\|  \left\| z_1 \right\|.
\end{aligned}$$

Moreover
\begin{align*}
g''(W_2+u)[v,z]&={e^{W_2+u}\over\int\limits_{\Omega }e^{W_2+u}}vz-{e^{W_2+u}\over\(\int\limits_{\Omega }e^{W_2+u}\)^2}v\int\limits_{\Omega }e^{W_2+u}z-{e^{W_2+u}\over\(\int\limits_{\Omega }e^{W_2+u}\)^2}z\int\limits_{\Omega }e^{W_2+u}v\\
 &-{e^{W_2+u}\over\(\int\limits_{\Omega }e^{W_2+u}\)^2}\int\limits_{\Omega }e^{W_2+u}vz+2{e^{W_2+u}\over\(\int\limits_{\Omega }e^{W_2+u}\)^3}\int\limits_{\Omega }e^{W_2+u}v
 \int\limits_{\Omega  }e^{W_2+u}z.\end{align*}

We use H\"older's inequalities with $ {1\over \alpha }+{1\over \beta }=1,$ $ {1\over a}+{1\over b }+{1\over c}=1,$ and ${1\over q }+ {1\over r }+{1\over s }+{1\over t }=1$
and we get
$$\begin{aligned}
&\left\|g''(W_2+u_2)[v_2,z_2]\right\|_p \nonumber\\ & \le {\|e^{W_2}\|_{pq}\|e^{u_2}\|_{pr}\|v_2\|_{ps}\|z_2\|_{pt}\over\|e^{W_2+u_2}\|_{1}}\nonumber\\ & + 2 {\|e^{W_2}\|^2_{pa}\|e^{ u_2}\|^2_{pb}\|v_2\|_{pc }\|z_2\|_{pc} \over\|e^{W_2+u_2}\|^2_{1}}\nonumber\\ &+{\|e^{W_2}\|_{p \alpha}\|e^{u_2}\|_{p \beta}\|e^{W_2}\|_{pq}\|e^{u_2}\|_{pr}\|v_2\|_{ps}\|z_2\|_{pt}\over\|e^{W_2+u_2}\|^2_{1}}\nonumber\\ &+2 {\|e^{W_2 }\|_{p \alpha}\|e^{u_2}\|_{p \beta}\|e^{W_2}\|^2_{pa}\|e^{ u_2}\|^2_{pb}\|v_2\|_{pc }\|z_2\|_{pc} \over\|e^{W_2+u_2}\|^3_{1}}\nonumber\\ &\hbox{(we use the continuity of $H^1_0(\Omega)\hookrightarrow L^p(\Omega)$, \eqref{ew2}, \eqref{l1} and Lemma \ref{tmt})} \nonumber \\
 & \le c_1 e^{ c_2\|u_2\|^2} \|v_2\|\|z_2\|  .\end{aligned}$$
It is important to point out that
\begin{equation}\label{l1}
\|e^{W_2+u_2}\|_{1}\ge c \ \hbox{provided $\|u_2\|$ is small enough.}
\end{equation}
Indeed,
we have
$$ \|e^{W_2+u}\|_{1}\ge \|e^{W_2 }\|_{1}- \|e^{W_2+u}-e^{W_2 }\|_{1},$$
$$\|e^{W_2+u}-e^{W_2 }\|_{1}=\int\limits_\Omega e^{W_2}|e^u-1|dx\le \int\limits_\Omega e^{W_2}|u|\le \|e^{W_2}\|_p\|u\|_{p \over p-1}\le \|e^{W_2}\|_p\|u\| $$
and the claim \eqref{l1} follows immediately by \eqref{ew2}.
\\
\end{proof}
  \begin{lemma}\label{B2} Let $\mathcal C \subset \mathcal F_k \Omega $ be a fixed compact set. For any $p,q>1$ there exist $\lambda_0>0,$ $r_0>0$  and $c_1,c_2>0$  such that for any $\la \in(0, \la_0)$, for any $\bxi\in\mathcal C$ and for any $\phi,\psi\in H^1_0(\Omega)\times H^1_0(\Omega) $ with $\|\phi\|,\|\psi\|\le r_0$
 \begin{equation}\label{B21}
\left \| {N} (\phi)\right\| \le c_1e^{c_2\|\phi \|^2}\lambda^{  \frac{1-pq}{pq}}\|\phi \|^2\end{equation}
  and
    \begin{equation}\label{B22}
\left\| {N} (\phi )-N(\psi)\right\| \le
c_1e^{c_2(\|\phi \|^2+\|\psi \|^2)}\lambda^{ \frac{1-pq}{pq}}
   \|\phi -\psi\|(\|\phi \|+\|\psi \|).\end{equation}
  \end{lemma}
\begin{proof} Let us remark that (\ref{B21}) follows by   choosing
$\psi =0 $ in (\ref{B22}) . Let us prove (\ref{B22}). First of all, we point
out that
$$\left\| {N} (\phi )-N(\psi)\right\| \le c_p\left\| F (W+\phi )-F (W+\psi )-F' (W)(\phi-\psi)\right\| _p.$$

We apply the mean value theorem (\cite{aprodi}, Theorem 1.8) to the map: $\varphi \mapsto f(\varphi + W_i) - f'(W_i)\, \varphi \in L^p(\Omega)$, with $\varphi \in H_0^1(\Omega)$. Then, there exists $\theta \in (0,1)$,

$$ \| f (W_i+ \phi_i )-f (W_i+ \psi_i )-f'(W_i) (\phi_i-\psi_i)\|_{p} \leq \| \[ f'(W_1+\theta \phi_i+(1-\theta)\psi_i)-f'(W_i)\](\phi_i-\psi_i)\|_{p}.$$

We apply again the mean value theorem to the map $\varphi \mapsto f'(\varphi + W_i)(\phi_i - \psi_i)$; there exists $\eta \in (0,1)$,

\begin{align*}& \| \[f'(W_1+\theta \phi_i+(1-\theta)\psi_i)-f'(W_i)\](\phi_i-\psi_i)\|_{p} \\ &\leq  \| f''(W_1+\eta (\theta \phi_i+(1-\theta)\psi_i))( \eta (\theta \phi_i+(1-\theta)\psi_i)) (\phi_1-\psi_1)\|_{p}.\end{align*}

We can argue in the same way to estimate the term:

$$ \| g (W_i+ \phi_i )-g (W_i+ \psi_i )-g'(W_i) (\phi_i-\psi_i)\|_{p}.$$

Lemma \ref{aux} allows us to conclude.

\end{proof}

\begin{lemma} \label{pixi} For any $\bxi$ in compact sets of $\mathcal{F}_k \Omega$ and any $\phi \in H^1_0(\Omega)$, there holds:

\begin{equation} \label{eq:pixi} \| \partial_{\xi_k^l} \Pi \ \phi \| = O(\la^{-\frac12}) \| \phi \|, \ \ \ \  \| \partial_{\xi_k^l} \Pi^{\perp} \ \phi \| = O(\la^{-\frac12}) \| \phi \|\end{equation}
\end{lemma}

\begin{proof}  Given any $\phi\in H^1_0(\Omega)$, we write $\Pi \, \phi$ in coordinates
$$\Pi \, \phi= \sum_{i, \ j}a_i^j PZ_i^j, $$ with $j=1,2$ and $i=1, \dots k$. Here the coefficients $a_i^j$ solve

  $$\langle\phi,\ PZ_{k}^l \rangle= \sum_{i, \ j} a_i^j \langle PZ_i^j,\ PZ_{k}^l \rangle. $$

In other words, the vector $a=(a_i^j)_{2k}$ solves the linear system:

  $$ A \cdot a = b,$$

with $b=(\langle\phi,\ PZ_i^j \rangle )_{2k}$ and $$A=(\langle PZ_i^j,\ PZ_{k}^l \rangle)_{2k \times 2k}. $$

Lemma \ref{stime-utili} implies that the elements in the diagonal of $A$ are of higher order than the others and $\|A \| =O(\la^{-1})$,
$\| A^{-1}\| = O(\la)$. Again by taking into account Lemma \ref{stime-utili}, we conclude:

\begin{equation} \label{a} a= A^{-1} \cdot b \Rightarrow \| a \| \leq O(\la^{\frac12}) \| \phi \|.\end{equation}

Computing now the derivative with respect to $\xi_k^l$, we obtain:

  $$ \partial_{\xi_k^l} a= A^{-1} \cdot \partial_{\xi_k^l} b  - \ A^{-1} \cdot \partial_{\xi_k^l} A \cdot A^{-1} \cdot b.$$

By Lemma \ref{stime-utili}, $\| \partial_{\xi_k^l} A \| \leq O(\delta^{-3})$. Therefore,

\begin{equation} \label{axi} \| \partial_{\xi_k^l} a \| \leq O(1) \| \phi \|. \end{equation}

Finally, observe that
$$ \partial_{\xi_k^l} \Pi \ \phi = \sum_{i, \ j}a_i^j \partial_{\xi_k^l} PZ_i^j + \partial_{\xi_k^l} a_i^j PZ_i^j.$$

Taking into account \eqref{a}, \eqref{axi} and Lemma \ref{stime-utili}, we conclude the first estimate of \eqref{eq:pixi}. The second follows from $\Pi^{\perp} \, \phi = \phi - \Pi \, \phi$.

\end{proof}

Now we are able to solve problem \eqref{goal}.
\begin{prop}\label{phi} Let $\mathcal C \subset \mathcal F_k \Omega $ be a fixed compact set. For any $\e>0$ there exists $\lambda_0>0,$    and $C>0$  such that for any $\la \in(0, \la_0)$ and for any $\bxi\in\mathcal C$ there exists a unique $\phi=\phi_{\bxi}\in{K^\perp}$ satisfying  \eqref{equ1} and
$$
\|     \phi\|\leq C\la^{\frac12-\e}
.$$ Moreover the map $\bxi\mapsto \phi_{\bxi}\in H^1_0(\Omega)\times H^1_0(\Omega)$ is $C^1$ and:

$$ \| \partial_{\xi_k^l} \phi \| \leq C \lambda^{-\e}.$$

 \end{prop}
\begin{proof} Equation \eqref{goal} can be solved via a contraction mapping argument. Indeed, in virtue of Proposition \ref{ex}, we can introduce the map
$$T(\phi):=L^{-1}\( N(     \phi)+  S(\phi)+    R\),\ \phi\in K^\perp.$$ By Lemma \ref{error}, Lemma \ref{aux0} and Lemma \ref{B2},
 it turns out to be a contraction map over the ball
\beq\label{ball}\left\{\phi\in K^\perp\ :\ \|\phi\|\le \Lambda|\log\la|\la^{2-p\over2p}\right\}\eeq provided $\Lambda$ is large enough and $\la$ is small enough.
Indeed, for any $\phi,\psi$ in the ball \eqref{ball}
$$\|T(\phi)\|\leq C|\log\la |\la^{\frac{1-pq}{pq}}\la^{\frac{2-p}{p}}+C|\log \la| \la^{\frac{2-p}{p}}<\Lambda|\log\la|\la^{2-p\over2p}$$ and
$$\|T(\phi)-T(\psi)\|\leq C|\log\la| \la^{\frac{1-pq}{pq}}\la^{\frac{2-p}{2p}}\|\phi-\psi\|<\|\phi-\psi\|$$provided that $p$ and $q$ are sufficiently close to 1.

We now consider the dependence of $\phi$ on $\bxi$.
We apply the Implicit Function Theorem to the function $\Phi:\mathcal{F}_k \Omega\times H^1_0(\Omega)\times H^1_0(\Omega)\to H^1_0(\Omega)\times H^1_0(\Omega)$ defined by
$$\Phi(\bxi, \phi)=\phi+\Pi^{\perp}\[W-i^*_p\Big(F(W+\Pi^\perp\phi\Big)\].$$
Indeed $\Phi(\bxi,\phi_{\bxi})=0$ and the linear operator: $\frac{\partial \Phi}{\partial \phi}(\bxi,\phi_{\bxi}):H^1_0(\Omega)\times H^1_0(\Omega)\to H^1_0(\Omega)\times H^1_0(\Omega)$ is given by

$$\begin{aligned}\frac{\partial \Phi}{\partial \phi}(\bxi,\phi_{\bxi})(\psi)&
=\psi-\Pi^\perp\[i^*_p\Big(F'(W+\phi_{\bxi})\Pi^\perp\psi \Big)\].\end{aligned}$$
We observe that $\frac{\partial \Phi}{\partial \phi}(\bxi,\phi_{\bxi})$ is a Fredholm operator.

By comparing $\frac{\partial \Phi}{\partial \phi}(\bxi,\phi_{\bxi})$ with the definition of $L$ in \eqref{elle}, we have then
$$\frac{\partial \Phi}{\partial \phi}(\bxi,\phi_{\bxi})(\psi)=\Pi(\psi)+L(\Pi^\perp\psi)-S(\Pi^\perp \psi)-\Pi^{\perp}\[i^*_p\Big(\big(F'(W+\phi_{\bxi})-F'(W)\big)\Pi^\perp \psi\Big)\]$$
by which, using Lemma \ref{ex} and Lemma \ref{aux0},
\beq\label{inj}\begin{aligned}&\left\|\frac{\partial \Phi}{\partial \phi}(\bxi,\phi_{\bxi})(\psi)\right\|\\ &\geq c\| \Pi(\psi)\|+c\|L(\phi_{\bxi})(\Pi^\perp\psi)\|-
\Big\|\Pi^\perp\[i^*_p\Big(\big(F'(W+\phi_\e)-F'(W)\big)\Pi^\perp \psi\Big)\]\Big\|-\|S(\Pi^\perp\psi)\|
\\ &\geq c\|\Pi(\psi)\|+ \frac{c}{|\log \lambda|} \|\Pi^\perp(\psi)\|-\|\big(F'(W+\phi_{\bxi})-F'(W)\big)\Pi^\perp \psi\|_{p}-\la^{\frac12-\e}\|\psi\|\\ &\geq \frac{c}{|\log \lambda|} \|\psi\|-\|\big(F'(W+\phi_{\bxi})-F'(W)\big)\Pi^\perp \psi\|_{p}.
\end{aligned}\eeq
Now, setting $\phi_{\bxi}:=(\phi_1,\phi_2)$, $\psi:=(\psi_1,\psi_2)$ and $\Pi^\perp:=(\Pi_1^\perp,\Pi_2^\perp)=(\Pi_1^\perp, id)$, we use Lemma \ref{aux} and  we compute for some $\theta_1\in (0,1)$
$$\begin{aligned}\|\big(f'(W_1+\phi_1)-f'(W_1)\big)\Pi_1^\perp \psi\|_p&=\|f''(W_1+\theta_1\phi_1)[\phi_1,\Pi_1^\perp \psi_1]\|_p\leq C\la^{\frac{1-pq}{pq}}\|\phi_{\bxi}\|\|\psi\|\\ &\leq C\la^{\frac{1-pq}{pq}+\frac{2-p}{2p}}\|\psi\|.
\end{aligned}$$
Similarly$$\begin{aligned}&\|\la\big(g'(W_2+\phi_2)-g'(W_2)\big) \psi_2\|_p=\|g''(W_2+\theta_2\phi_2)[\phi_2, \psi_2]\|_p\leq C\|\phi_{\bxi}\|\|\psi\|\leq C
\la^{\frac{2-p}{2p}}\|\psi\|.
\end{aligned}$$
We take $p$, $q$ sufficiently close to 1 and combine the above two estimates with \eqref{inj}, to conclude that $$\left\|\frac{\partial \Phi}{\partial \phi}(\bxi,\phi_{\bxi})(\psi)\right\|\geq \frac{C}{|\log \la |} \|\psi\|.$$ This implies the invertibility of the operator $\frac{\partial \Phi}{\partial \phi}(\bxi,\phi_{\bxi})$, and moreover
$$\left ( \frac{\partial \Phi}{\partial \phi}(\bxi,\phi_{\bxi}) \right )^{-1} \leq C |\log \la |.$$

By the implicit function theorem, the map $\bxi \mapsto \phi_{\bxi}$ is $C^1$ and

$$ \frac{\partial \phi}{\partial \bxi} = - \left ( \frac{\partial \Phi}{\partial \phi}(\bxi,\phi_{\bxi}) \right )^{-1} \left ( \frac{\partial \Phi}{\partial \bxi}  \right ).$$

So we need to estimate $\frac{\partial \Phi}{\partial \bxi}$.

\begin{align*}\Big\|\frac{\partial \Phi}{\partial \bxi} \Big\|&\leq
\|\partial_{\bxi_i} \Pi^{\perp} \phi \| \|W-i^*_p\big(F(W+\phi_{\bxi})\big)\|+\|\partial_{\bxi_i}W-i^*_p\big(F'(W+\phi_{\bxi})(\partial_{\bxi_i} W+\partial_{\bxi_i}(\Pi^\perp)\phi_{\bxi})\big)\|
\\ &\leq C \la^{-1/2}\|-\Delta W-F(W+\phi_{\bxi})\|_p +C \|-\partial_{\bxi_i}\Delta W-F'(W+\phi_{\bxi})[\partial_{\bxi_i} W+\partial_{\bxi_i}(\Pi^\perp)\phi_{\bxi}]\|_{p}
\\ &\leq C\la^{-1/2}\Big( \|-\Delta W-F(W)\|_p + \| F(W)-F(W+\phi)\|_p \Big ) \\ & + C  \|-\partial_{\bxi_i}\Delta W-F'(W) \partial_{\bxi_i} W \|_{p}
+ C \| [F'(W+\phi)- F'(W)](\partial_{\bxi_i} W)\|_{p} \\ & + C \| F'(W+\phi) \partial_{\bxi_i} \Pi^{\perp} \phi\|_{p}.
\end{align*}

By using the mean value Theorem (Theorem 1.8 of \cite{aprodi}) as in the proof of Lemma \ref{B2}, and taking into account Lemma \ref{aux}, we conclude:

$$ \| F(W)-F(W+\phi)\|_{p} \leq C \la^{\frac{1-pq}{pq}}\| \phi \| \leq C \la^{\frac{1-pq}{pq} + \frac{2-p}{2p}}.$$

Similarly,

$$\| [F'(W+\phi)- F'(W)](\partial_{\bxi_i} W)\|_{p} \leq  C \la^{\frac{1-pq}{pq}} \| \phi \| \| \partial_{\bxi_i} W\| \leq C \la^{\frac{1-pq}{pq}+\frac{2-p}{2p}-\frac12},$$
where we have used estimate \eqref{der2} in Lemma \ref{stime-utili}.

Moreover,

$$  \| F'(W+\phi) \partial_{\bxi_i} \Pi^{\perp} \phi\|_{p} \leq  C \la^{\frac{1-pq}{pq}+\frac{2-p}{2p}-\frac12}   .$$

Therefore, we just need to estimate the $L^p$ norms of the terms:

$$ -\Delta W-F(W) , \ \ -\partial_{\bxi_i} \Delta W-F'(W) \partial_{\bxi_i} W . $$

But these are, respectively, $\tilde{R}$ and $\nabla_{\bxi} \tilde{R}$ as defined in \eqref{Rtilde}. And their $L^p$ norms have been estimated in Lemma \ref{error}, so we finish the proof.
\end{proof}

Setting
$$u_1=W_1+\phi_1,\qquad u_2=W_2+\phi_2,$$ where  $\phi_{\bxi}=(\phi_1,\phi_2)$ is provided by Proposition \ref{phi}, then $u=(u_1,u_2)$ satisfies

\beq\label{eqq}\left\{\begin{aligned}&\Delta u_1+2\la f(u_1)-\rho_2 g(u_2)=\sum\limits_{j=1,2\atop i=1,\dots,k}c_{ij}Z^j_ie^{w_i}\\ &\Delta u_2+2\rho_2 g(u_2)-\la f(u_1)=0
\\ &\phi_1,\phi_2\in H^1_0(\Omega),\;\; \into \nabla\phi_1\nabla PZ^j_idx=0 \ j=1,2,\ i=1,\dots,k.
\end{aligned}\right.\eeq
where the constants  $c_{ij}$ verify \beq\label{ciij}|c_{ij}|\leq C\la \eeq according to  \eqref{combi2}.
Therefore the following identities hold:
\beq\label{hem1} \frac23\Delta u_1+\frac13\Delta u_2+\la f(u_1)=\frac23\sum\limits_{j=1,2\atop i=1,\dots,k}c_{ij}Z^j_ie^{w_i} \eeq \beq\label{hem2}\frac23\Delta u_2+\frac13\Delta u_1+\rho_2g(u_2)=\frac13\sum\limits_{j=1,2\atop i=1,\dots,k}c_{ij}Z^j_ie^{w_i} \eeq

Let us consider the energy functional associated to the system \eqref{s}:
\begin{equation}\label{energy}
J(u_1,u_2):={1\over3}Q(u_1,u_2)-\la\int\limits_\Omega e^{u_1(x)}dx-\rho_2\log \int\limits_\Omega e^{u_2(x)}dx
\end{equation}
where $$Q(u_1,u_2)=\int\limits_\Omega\(|\nabla u_1|^2+|\nabla u_2|^2+\nabla u_1\nabla u_2\)dx.$$

Next lemma  concerns the relation between the critical
points of $M$ and  those of the energy functional $J $.

\begin{lemma}\label{relation}
Let $\bxi  \in \mathcal F_k \Omega $ be a critical point of $\bxi\mapsto J(W+\phi_{\bxi}) $.
Then, provided that $\e>0$ is sufficiently small, the
corresponding function $ u= W + \phi_{\bxi}$ is a solution of \eqref{s}.
\end{lemma}
\begin{proof}
According to \eqref{eqq}, we  find a solution to the original problem \eqref{s} if $\bxi$ is such that $c_{ij}=c_{ij}(\bxi)=0$ for $j=1,2$ and $i=1,\dots,k$.

Let us we fix $q>1$.  Using that $$\|Z_i^je^{w_i}\|_{q}=
O\(\la^{\frac{2-3q}{2q}}\),$$ by Proposition \ref{phi}
for any $i=1,\ldots, k$ and $j=1,2$ we have
\beq\label{remhem}\into \Big|\partial_{\xi^\ell_h} \phi_1 Z^j_{i}\Big|e^{w_i}dx=O(\la^{-\e})\|Z^j_{i}e^{w_i}\|_q=O(\la^{-\e+\frac{2-3q}{2q}})=o\Big(\frac1\la\Big).\eeq
provided that $q$ is sufficiently close to $1$. Similarly \beq\label{remhem1}\into \Big|\partial_{\xi^\ell_h} \phi_2 Z^j_{i}\Big|e^{w_i}dx=o\Big(\frac1\la\Big).\eeq
Let $\bxi\in\mathcal F_k \Omega $  be a critical point of $\bxi\mapsto J(W+\phi_{\bxi}) $.
By differentiating the function $J(W+\phi_{\bxi}) $ with respect to $\bxi$ and using the ${\cal C}^1$ regularity of the map $\xi\mapsto\phi_{\bxi}\in (H^1_0(\Omega))^2$, we get for $\ell=1,2$ and $h=1,\dots,k$
$$\into \Big(\frac23\Delta u_1+\frac13\Delta u_2+\la f(u_1)\Big)\partial_{ \xi^\ell_h} (W_1+\phi_1)dx+\into \Big(\frac23\Delta u_2+\frac13\Delta u_1+\rho_2g(u_2)\Big)\partial_{ \xi^\ell_h} (W_2+\phi_2)dx=0.$$
which is equivalent, by \eqref{hem1}-\eqref{hem2}, to
\beq\label{diag}\sum\limits_{j=1,2\atop i=1,\dots,k}c_{ij}Z^j_ie^{w_i}\Big(2\partial_{\xi^\ell_h} (W_1+\phi_1) +\partial_{\xi^\ell_h} (W_2+\phi_2) \Big)dx=0
\quad \ell=1,2,\ h=1,\dots,k.\eeq We observe that by \eqref{pzi}
\beq\label{vado}\partial_{ \xi^\ell_h} W_1=4PZ_i^j+O(1), \quad \partial_{ \xi^\ell_h} W_2=-2PZ_i^j+O(1).\eeq
Therefore, combining \eqref{diag} with \eqref{der3} and \eqref{remhem}-\eqref{remhem1},  we get
$$c_{\ell h}+o\(1\)\sum\limits_{j\not=\ell\atop i\not=h}c_{ij}=0\ \hbox{for any $\ell=1,2$ and $h=1,\dots,k,$}$$
so   the system \eqref{diag} is diagonal dominant and then we achieve that all the $c_{ij}$'s are zero. That concludes the proof.

\end{proof}

\section{The reduced energy and proof of Theorem \ref{teo}}
 The first  purpose of this section is to give an asymptotic estimate of $J(W_1, W_2)$, where  $(W_1,W_2)$ is the approximate solution  defined in \eqref{ans} and $J$ is given in \eqref{energy}.

 \begin{prop}\label{asymp} The following holds:
 $$J(W_1,W_2)=-4 \pi k\log \lambda +\Lambda(\bxi)
-8\pi (1-\log 2) + o(1),$$
 ${\cal C}^1$ uniformly with respect to $\xi$ in compact sets of $\Omega$, where $\Lambda$ is defined in \eqref{Lambda}.
 \end{prop}

\begin{proof}

By the definition of $W_i$ we have:
$$ Q(W_1,W_2)= \frac 1 4 \left \{ \into |\nabla z(x,\bxi)|^2 + \sum_{i=1}^k|\nabla P w_i |^2 +\sum_{i,j=1\atop i\neq j}^k \into \nabla Pw_i\nabla Pw_j - \sum_{i=1}^k\nabla P w_i \cdot \nabla z(x,\bxi) \right \}.$$
Moreover, by Lemma \ref{stima-E},
\begin{equation} \label{1} \begin{aligned}\frac 1 4  \into |\nabla z(x,\bxi)|^2 - \rho \log \into e^{W_2}&
=\frac 1 4 \int_{\Omega} |\nabla z(x,\bxi)|^2 - \rho \log \into h(x,\bxi) e^{z(x,\bxi)} + O(\la).\end{aligned}\end{equation}
Again by Lemma \ref{stima-E}
 it is easy to check that:
\begin{equation} \label{2} \lambda \into e^{W_1}=\frac12\sum_{i=1}^k \into e^{w_i}+o(1)=\frac12\sum_{i=1}^k
\int_{\frac{\Omega-\xi_i}{\de_i}}\frac{8}{(1+|y|^2)^2} dy +o(1)= 4 \pi k+o(1) \end{equation}
where we have used  that $\int_{\R^2}\frac{1 }{(1 + |y|^2)^2} \,dy=  \pi. $
We continue with the  estimate of the terms $\into |\nabla P w_i |^2$, $\into  \nabla P w_i \cdot \nabla z(x,\bxi).$  Integrating by parts,

\begin{equation} \label{33}\begin{aligned}  \frac 1 4 \into  \nabla P w_i\cdot \nabla z(x,\bxi) &= \frac 1 4 \into  e^{w_i} z(x,\bxi)=\frac14 \into \frac{8\de_i^2}{(\de_i^2+|x-\xi|^2)^2}z(x,\bxi) dx\\ &=2 \int_{\frac{\Omega-\xi_i}{\de_i}}\frac{z(\de_i y+\xi_i,\bxi)}{(1+|y|^2)^2} dy
=2 \pi z(\xi_i,\bxi)+o(1).\end{aligned}\eeq

We now estimate the  term:

$$ \begin{aligned}\frac 1 4  \into |\nabla P w_i |^2& = \frac 1 4  \into e^{w_i} \Big( \log \frac{1}{(\delta_i^2 + |x-\xi_i|^2)^2} + 8 \pi H(x,\xi_i) + O(\la) \Big) dx\\ &=
\frac 1 4  \into e^{w_i} 8 \pi H(x,\xi_i)\, dx -\frac 1 2 \into e^{w_i}  \log (\delta_i^2 + |x-\xi_i|^2)\, dx + O(\la). \end{aligned}$$
Arguing as in \eqref{33}, we conclude that:
\begin{equation} \label{4}  \frac 1 4  \into e^{w_i} 8 \pi H(x,\xi_i) = 16 \pi^2 H(\xi_i, \xi_i)+o(1). \end{equation}
For the second term, we make the change of variables $x-\xi_i= \delta_i y$, to get:
$$ \begin{aligned}&-\frac 1 2  \into e^{w_i}  \log (\delta_i^2 + |x-\xi_i|^2)\, dx=-\frac 1 2  \into   \frac{8 \delta_i^2}{(\delta_i^2 + |x-\xi_i|^2)^2} \log (\delta_i^2 + |x-\xi_i|^2)\, dx  \\ &
=-\frac 1 2  \int_{\frac{\Omega-\xi_i}{\de_i}}  \frac{8 }{(1 + |y|^2)^2} \log (\delta_i^2(1 + |y|^2))\,dy\\ &=
-\frac 1 2 \log \delta_i^2 \int_{\frac{\Omega-\xi_i}{\de_i}}  \frac{8 }{(1 + |y|^2)^2} \,dy - \frac 1 2  \int_{\frac{\Omega-\xi_i}{\de_i}}  \frac{8 }{(1 + |y|^2)^2}  \log (1 + |y|^2) \,dy\\ &
=
- 4 \pi \log \delta_i^2 -4\pi +o(1)
\end{aligned}$$
since $\int_{\R^2}  \frac{8 }{(1 + |y|^2)^2}  \log \frac{1}{(1 + |y|^2)^2} \,dy=\pi$. Therefore, recalling the definition of $\de_i$ in \eqref{de}  we obtain \beq\label{5}\begin{aligned}&\frac 1 4  \into |\nabla P w_i |^2=16 \pi^2 H(\xi_i, \xi_i)- 4\pi \log \delta_i^2 -4\pi +o(1)\\ &=-16 \pi^2 H(\xi_i,\xi_i)-32\pi^2\sum_{j\neq i} G(\xi_i,\xi_j)+2\pi z(\xi_i,\bxi)-4\pi\log \la+8\pi\log 2-4\pi+o(1).\end{aligned}\eeq

Finally, for $i\neq j$, by \eqref{exp}, reasoning as in \eqref{33},
\beq\label{6}\begin{aligned}\frac14\into \nabla Pw_i\nabla Pw_j&=\frac14\into e^{w_i}\Big(\log \frac{1}{(\delta_j^2+|x-\xi_j|^2)^2}+8\pi H(x,\xi_j)+O(\la)\Big)\\ &=
2\pi \Big(\log \frac{1}{|\xi_i-\xi_j|^4}+8\pi H(\xi_i,\xi_j)\Big)+o(1)=16\pi^2G(\xi_i,\xi_j)+o(1)
\end{aligned}\eeq

Putting together equations \eqref{1}, \eqref{2}, \eqref{33}, \eqref{4}, \eqref{5}, we conclude with the ${\cal C}^0$ estimate.

We are going to estimate the error term in the ${\cal C}^1$ sense. By using Lemma \ref{stima-E} and \eqref{vado}
we get
\beq\label{c1uno}\begin{aligned} \partial_{\xi_i^j} J(W)&=\into \Big(-\frac23\Delta W_1-\frac13\Delta W_2-\la f(W_1)\Big)\partial_{\xi_i^j} W_1 dx\\ &\;\;\;\;+\into \Big(-\frac23\Delta W_2-\frac13\Delta W_1-\rho_2g(W_2)\Big)\partial_{\xi_i^j} W_2dx
\\ &=
4\into\Big(\frac12E+\frac14 E_0\Big)Z_i^j dx+o(1)
= 2 \into EZ_i^jdx+o(1)
\\ &=2\sum_{h=1}^k\into e^{w_h} Z_i^j dx-4\la \into e^{W_1}Z_i^j dx +o(1)
.\end{aligned}\eeq
For any $h\neq i$ we have, reasoning as in \eqref{33},
\beq\label{c1due} \into e^{w_h} Z_i^jdx=\into {8\de_h^2\over\(\de_h^2+|x-\xi_h|^2\)^2}{x_j-\xi_i^j\over \de_i^2+|x-\xi_i|^2}dx  8\pi{\xi_h^j-\xi_i^j\over |\xi_h-\xi_i|^2}+o(1),
\eeq while, for $i=j$,

\beq\label{c2tre}\begin{aligned}\into e^{w_i} Z_i^jdx&=\into {8\de_i^2\over\(\de_i^2+|x-\xi_i|^2\)^2}{x_j-\xi_i^j\over \de_i^2+|x-\xi_i|^2}dx\\ &=\frac{1}{\delta_i}\int_{\frac{\Omega-\xi_i}{\de_i}}{8\over\(1+|y|^2\)^2}{y_j\over 1+| y|^2}dx= o(1).
\end{aligned}\eeq
Let $\eta>0$ be such that $|\xi_i-\xi_j|\geq 2\eta$ and $\di(\xi_i,\partial\Omega)\geq 2\eta$.
Then, for $h\neq i$, using the change of variable $x=\delta_h y+\xi_h$,
\beq\label{c2quattro}\begin{aligned}&\int_{B(\xi_h, \eta)} 2\la e^{W_1} Z_i^j\\ &=2\la\int_{B(\xi_h, \eta)} e^{   8\pi \sum_{l=1}^kH(x,\xi_l) -{1\over2}z(x,\bxi)+O(\la) }  {x_j-\xi_i^j\over \de_i^2+|x-\xi_i|^2 }\prod_{l=1}^k{1\over\(\de_l^2+|x-\xi_l|^2\)^2}dx
\\ &=2\frac{\la}{\delta_h^2}\int_{\R^2} {e^{   8\pi \sum_{l=1}^kH(\xi_h,\xi_l)-{1\over2}z(\xi_h,\bxi)  }\over(1+|y|^2)^2}{\xi_h^j-\xi_i^j\over |\xi_h-\xi_i|^2}\prod_{l=1\atop l\neq h}^k {1\over|\xi_h-\xi_l|^4}  dx+o(1)
\\ &=2\frac{\la}{\de_h^2}\int_{\R^2}{e^{   8\pi H(\xi_h,\xi_h)+8\pi\sum_{l \neq h}G(\xi_h,\xi_l)  -{1\over2}z(\xi_h,\bxi)  }\over(1+|y|^2)^2}{\xi_h^j-\xi_i^j\over |\xi_h-\xi_i|^2}dx+o(1)
\\ &=8\pi{\xi_h^j-\xi_i^j\over |\xi_h-\xi_i|^2}+o(1)
\end{aligned}\eeq where in the last inequality we have used  the choice of $\de_h$ in \eqref{de}. Similarly
$$\begin{aligned}&\int_{B(\xi_i, \eta)} 2\la e^{W_1} Z_i^j\\ &=2\la\int_{B(\xi_i, \eta)} e^{   8\pi \sum_{l=1}^kH(x,\xi_l)+ -{1\over2}z(x,\bxi)+O(\la)}  {x_j-\xi_i^j\over \de_i^2+|x-\xi_i|^2 }\prod_{l=1}^k{1\over\(\de_l^2+|x-\xi_l|^2\)^2}dx
\\ &=2\frac{\la}{\de_i^3}\int_{B(0, \frac{\eta}{\delta_i})}{e^{   8\pi \sum_{l=1}^kH(\delta_i y+\xi_i,\xi_l)-{1\over2}z(\delta_i y+\xi_i,\bxi)+O(\la)  }\over(1+|y|^2)^2}{y_j\over (1+|y|^2)}\prod_{l=1\atop l\neq i}^k {1\over|\delta_i y+\xi_i-\xi_l|^4} dy+o(1)
\\ &=2\frac{\la}{\de_i^3}\int_{B(0, \frac{\eta}{\delta_i})}{e^{   8\pi H(\delta_i y+\xi_i,\xi_i)+8\pi\sum_{l \neq i}G(\delta_i y+\xi_i,\xi_l)    -{1\over2}z(\delta_i y+\xi_i,\bxi)+O(\la)  }\over(1+|y|^2)^2}{y_j\over (1+|y|^2)}+o(1).
\end{aligned}$$
Next we set  $\gamma(x,\bxi):=8\pi H(x,\xi_i)+8\pi\sum_{l \neq i}G(x,\xi_l)-\frac12z (x,\bxi)$. Then we use the choice of $\de_i$ in \eqref{de} and by mean value theorem we obtain
\beq\label{c1cinque}\begin{aligned}\int_{B(\xi_i, \eta)} 2\la e^{W_1} Z_i^j
&=2{\la\over \de_i^3}\int\limits_{B(0, \frac{\eta}{\delta_i})}{1 \over \(1+|y|^2\)^2}  e^{   \gamma (\de_i y+\xi_i,\xi) +O(\la) } {y_j\over 1+|y|^2}dy+o(1)
\\ &={8\over \de_i}\int\limits_{B(0,\frac{\eta}{\delta_i})}{ y_j\over \(1+|y|^2\)^3}  e^{   \gamma (\de_i y+\xi_i,\bxi)-\gamma(\xi_i,\bxi)+O(\la) }  dy+o(1)\\ &
={8\over \de_i}\int\limits_{B(0,\frac{\eta}{\delta_i})}{ y_j\over \(1+|y|^2\)^3}\( \frac{\partial\gamma}{\partial x} (\xi_i,\bxi)\cdot \de_i y+O(\la|y|^2)+O(\la)\) dy +o(1)\\ &
= 8\frac{\partial\gamma}{\partial x_j}(\xi_i,\bxi)  \int\limits_{\rr^2}{y_j^2\over \(1+|y|^2\)^3}dy+o(1)=2\pi \frac{\partial\gamma}{\partial x_j}(\xi_i,\bxi) +o(1).
\end{aligned}\eeq
On the other hand

\beq\label{c1sei}\begin{aligned}&\bigg|\int_{\Omega\setminus \cup_{h=1}^kB(\xi_h, \eta)}2\la e^{W_1} Z_i^jdx\bigg|\leq C\la\int_{\Omega\setminus \cup_{h=1}^kB(\xi_h, \eta)} e^{\sum_{l=1}^k P w_l} |Z_i^j|dx\\ &\leq C\la \int_{\Omega\setminus B(\xi_i, \eta)} e^{P w_i} |Z_i^j|
\leq C\la \int_{\Omega\setminus B(\xi_i, \eta)}  {1\over\(\de_i^2+|x-\xi_i|^2\)^2}{|x_j-\xi_i^j|\over \de_i^2+|x-\xi_i|^2}=o(1)\end{aligned}
\eeq
and the thesis follows by combining \eqref{c1uno}-\eqref{c1sei} once we have observed that, by Lemma \ref{grad},  $\frac{\partial\gamma}{\partial x_j}(\xi_i,\bxi) =-\frac{1}{4\pi}\partial_{\xi_i^j}\Lambda(\bxi)$.

\end{proof}
For $\la>0$ sufficiently small we consider the reduced functional $$\tilde J(\bxi)=J(u_1,u_2)=J(W_1+\phi_1, W_2+\phi_2)$$ where $\phi_{\bxi}=(\phi_1,\phi_2)$ has been constructed in Lemma \ref{phi}. The next proposition contains the key expansions of $\tilde J$.
\begin{prop}\label{expphi} The following holds
$$\tilde J(\bxi)=-4 \pi k\log \lambda +\Lambda(\bxi)
-8\pi (1-\log 2) + o(1),$$
 ${\cal C}^1$ uniformly with respect to $\xi$ in compact sets of ${\cal F}_k(\Omega)$, where $ \Lambda(\bbm[\xi])$ has been defined in \eqref{Lambda}.

\end{prop}
\begin{proof}
We compute
\beq\label{czero}\begin{aligned}\tilde J(\bxi)&= J(W_1, W_2)-\frac12\bigg(\into \Big(\sum_{i=1}^k\Delta Pw_i+2\la e^{W_1}\Big) \phi_1+\into \Delta z(x,\bxi) \phi_2\bigg)
\\ &\;\;\;\;-\la \into \big(e^{W_1+\phi_1}-e^{W_1}-e^{W_1}\phi_1\big)-\rho_2\log\into e^{W_2+\phi_2}+\rho \log \into e^{W_2}+o(1).
\end{aligned}\eeq
Next by Lemma \ref{stima-E} we estimate:
$$\into \Big(\sum_{i=1}^k\Delta Pw_i+2\la e^{W_1}\Big) \phi_1=-\into E_1\phi_1= o(1).$$
Moreover,  for some $\theta_1\in (0,1) $, using Lemma \ref{aux}, we have
$$\la\into \big(e^{W_1+\phi_1}-e^{W_1}-e^{W_1}\phi_1\big)=\la\into e^{W_1+\theta_1\phi_1} \phi_1^2=O(\la^{\frac{1-pq}{pq}+1-2\e})=o(1)
$$ provided that $p,\,q$ are sufficiently close to 1.
Furthermore, for some $\theta_2\in (0,1)$, using also \eqref{l1}, we get
$$\begin{aligned}\log\into e^{W_2+\phi_2}-\log \into e^{W_2}=\frac{1}{\into e^{W_2+\theta_2\phi_2}}\into e^{W_2+\theta_2\phi_2}\phi_2=O(\|\phi_2\|)=o(1).\end{aligned}$$ Finally we immediately obtain $$\into \Delta z(x,\bxi) \phi_2=O\Big(\into |\phi_2|\Big)=o(1).$$
By inserting the above estimates into \eqref{czero} Proposition \ref{asymp} gives the ${\cal C}^0$ estimate.

 In order to prove that the expansion actually holds in the ${\cal C}^1$ sense,
we compute
$$\begin{aligned}\partial_{\xi_i^j} \tilde J(\bxi)&=
\partial_{\xi_i^j}J(W_1,W_2)+\into \Big(-\frac23 \Delta\phi_1-\frac13\Delta\phi_2-\la (f(W_1+\phi_1)-f(W_1))\Big)\partial_{\xi_i^j} W_1\\ &
\;\;\;\;+\into \Big(-\frac23\Delta \phi_2-\frac13\Delta \phi_1-\rho_2(g(W_2+\phi_2)-g(W_2))\Big)\partial_{\xi_i^j} W_2dx\\ &\;\;\;\;
+\into \Big(-\frac23\Delta u_1-\frac13\Delta u_2-\la f(u_1)\Big)\partial_{\xi_i^j}\phi_1dx\\ &\;\;\;\;+\into \Big(-\frac23\Delta u_2-\frac13\Delta u_1-\rho_2g(u_2)\Big)\partial_{\xi_i^j} \phi_2dx\end{aligned}$$
Then, by \eqref{ciij}, \eqref{hem1}-\eqref{hem2}   and \eqref{remhem}-\eqref{remhem1} we deduce \beq\label{vado4}\begin{aligned} \partial_{\xi_i^j}\tilde J(\bxi)&=
\partial_{\xi_i^j} J(W_1,W_2)+\into \Big(-\frac23 \Delta\phi_1-\frac13\Delta\phi_2-\la (f(W_1+\phi_1)-f(W_1))\Big)\partial_{\xi_i^j} W_1\\ &
\;\;\;\;+\into \Big(-\frac23\Delta \phi_2-\frac13\Delta \phi_1-\rho_2(g(W_2+\phi_2)-g(W_2))\Big)\partial_{\xi_i^j} W_2dx+o(1)
.\end{aligned}\eeq

By Lemma \ref{aux} and \eqref{der2} for some $\theta_1\in (0,1)$ we get\beq\label{vado5}\begin{aligned}\la \into\big(f(W_1+\phi_1)-f(W_1)-f'(W_1)\phi_1)\big)\partial_{\xi_i^j} W_1&= O\Big(\la^{-\frac12}\| \la f''(W_1+\theta_1\phi_1)\phi_1^2\|_p\Big)\\ &=
 O(\la^{\frac{1-pq}{pq}+1-2\e-\frac12})=o(1)\end{aligned}\eeq provided that $p$  is sufficiently close to 1.
 Since $\into e^{w_i}\phi_1Z_i^j=-\into\nabla \phi_1\nabla PZ_i^j=0$, then   $$\into e^{w_i}\phi_1PZ_i^j =\into e^{w_i}\phi_1Z_i^j+o(1)=o(1),$$ while, for $h\neq i$, by \eqref{cii},
 $$\into e^{w_h}\phi_1 PZ_i^j=O(\la^{\frac{1-p}{p}+\frac12-\e})=o(1)$$ provided that $p$ is sufficiently close to 1,
   by which, using Lemma \ref{stima-E}, \eqref{der3} and  \eqref{vado},
\beq\label{vado6}\la\into f'(W_1)\phi_1\partial_{\xi_i^j}W_1=\frac12\sum_{h=1}^k\into e^{w_h}\phi_1 PZ_i^j +O(\la^{\frac{2-p}{2p}-\e})=o(1).
\eeq
Next we choose $p,q>1$ such that $\frac1p+\frac1q=1$ and $1<q<2$. Then, according to \eqref{vado},  $\|\partial_{\xi_i^j} W_2\|_q=O(1)$. Consequently,  again by Lemma \ref{aux}, for some $\theta_2\in (0,1)$,
\beq\label{vado7}\into(g(W_2+\phi_2)-g(W_2))\partial_{\xi_i^j}W_2dx=O(\| g'(W_2+\theta_2\phi_2)\phi_2\|_p)=O(\la^{\frac12-\e})=o(1).
\eeq
We observe that \beq\label{ww}\partial_{\xi_i^j} W_1=4 P Z_i^j+O(1),\qquad \partial_{\xi_i^j} W_2=-2 P Z_i^j+O(1),\eeq ${\cal C}^1$-uniformly for $x\in\overline\Omega$ and $\bxi$ on compact sets of ${\cal F}_k(\Omega)$.
Moreover, by \eqref{ww}, and recalling that $ \into\nabla \phi_1\nabla PZ_i^j=0$,
\beq\label{vado8} \into \Delta\phi_1\partial_{\xi_i^j} W_1=-\into  \nabla\phi_1\nabla\partial_{\xi_i^j}W_1=-\into\nabla \phi_1(\nabla PZ_i^j+O(1))=O\Big(\into |\nabla \phi_1|\Big)=o(1).\eeq
Similarly \beq\label{vado81} \into \Delta\phi_1\partial_{\xi_i^j} W_2=o(1).\eeq
Finally we have $-\frac13\partial_{\xi_i^j} W_1-\frac23 \partial_{\xi_i^j} W_2=-\frac12z(x,\bxi),$ by which
\beq\label{vado9} \into \Delta\phi_2\Big(-\frac13\partial_{\xi_i^j} W_1-\frac23\partial_{\xi_i^j} W_2\Big)=\frac12\into\nabla\phi_2\nabla z(x,\bxi)=O\Big(\into |\nabla \phi_2|\Big)=o(1). \eeq
By inserting \eqref{vado5}-\eqref{vado9} into \eqref{vado4}, the thesis follows by Proposition \ref{asymp}. \end{proof}

\noindent{\bf Proof of Theorem \ref{teo} completed}. Let ${\cal K}\subset {\cal F}_k(\Omega)$ be a $C^1$-stable set of critical points of $\Lambda$. Then, according to Proposition \ref{expphi}, for $\la>0$ sufficiently small  let $\bxi_\la$ be a critical point of $\tilde J$ such that $\di(\bxi,{\cal K})\to 0$.  By Lemma \ref{relation} $u_\la=W+\phi_{\bxi_\la}$ solves problem \eqref{s}. Consequently $u_\la$ provides a solution to the original problem \eqref{eq:e-1} with  $$\rho_1=\rho_{1,\la}=\la\into e^{u_1}.$$ Then, by \eqref{2} and Lemma \ref{aux}
$$\rho_{1,\la}=\la \into e^{W_1+\phi_1}=\la\into e^{W_1}+o(1)=4k\pi +o(1).$$
\bigskip
\medskip

\small{\noindent{\bf Acknowledgments.}  The authors are grateful to D. Bartolucci and P. Esposito for some useful discussions.
}

\end{document}